\documentclass[a4paper,10pt]{article}
\usepackage[utf8x]{inputenc}
\usepackage{tracefnt,amsmath,tabu,array}
\usepackage{amssymb,graphicx,setspace,amsfonts,amsbsy}
\usepackage{pifont,latexsym,ifthen,amsthm,rotating,calc,textcase,booktabs,cancel,slashed}

\newtheorem{theorem}{Theorem}[section]
\newtheorem{lemma}[theorem]{Lemma}
\newtheorem{corollary}[theorem]{Corollary}
\newtheorem{proposition}[theorem]{Proposition}
\newtheorem{definition}[theorem]{Definition}

\newtheorem{conjecture}[theorem]{Conjecture}
\newtheorem{remark}[theorem]{Remark}
\newcommand{\filledbox}{\leavevmode
  \hbox to.77778em{%
  \hfil\vbox to.675em{\hrule width.6em height.6em}\hfil}}

\newcommand{\Rm}{{\mathbb R}}

\newcommand{\eps}{\varepsilon}

\begin{document}
\tabulinesep=1.0mm
\title{Inward/outward Energy Theory of Non-radial Solutions to 3D Semi-linear Wave Equation\footnote{MSC classes: 35L71, 35L05; This work is supported by National Natural Science Foundation of China Programs 11601374, 11771325}}

\author{Ruipeng Shen\\
Centre for Applied Mathematics\\
Tianjin University\\
Tianjin, China}

\maketitle

\begin{abstract}
  The topic of this paper is a semi-linear, energy sub-critical, defocusing wave equation $\partial_t^2 u - \Delta u = - |u|^{p -1} u$ in the 3-dimensional space 
  with $3\leq p<5$. We generalize inward/outward energy theory and weighted Morawetz estimates for radial solutions to the non-radial case. As an application we show that if $3<p<5$ and $\kappa>\frac{5-p}{2}$, then the solution scatters as long as the initial data $(u_0,u_1)$ satisfy   
 \[
  \int_{\Rm^3} (|x|^\kappa+1)\left(\frac{1}{2}|\nabla u_0|^2 + \frac{1}{2}|u_1|^2+\frac{1}{p+1}|u_0|^{p+1}\right) dx < +\infty.
 \]
If $p=3$, we can also prove the scattering result if initial data $(u_0,u_1)$ are contained in the critical Sobolev space and satisfy the inequality
\[
 \int_{\Rm^3} |x|\left(\frac{1}{2}|\nabla u_0|^2 + \frac{1}{2}|u_1|^2+\frac{1}{4}|u_0|^{p+1}\right) dx < +\infty.
\]
These assumptions on the decay rate of initial data as $|x| \rightarrow \infty$ are weaker than previously known scattering results.
\end{abstract}

\section{Introduction}

\subsection{Background}

We consider the Cauchy problem of the defocusing semi-linear wave equation in 3-dimensional case 
\[
 \left\{\begin{array}{ll} \partial_t^2 u - \Delta u = - |u|^{p-1}u, & (x,t) \in \Rm^3 \times \Rm; \\
 u(\cdot, 0) = u_0; & \\
 u_t (\cdot,0) = u_1. & \end{array}\right.\quad (CP1)
\]
\paragraph{Critical Sobolev spaces} We define $s_p = 3/2-2/(p-1)$ and call $\dot{H}^{s_p}\times \dot{H}^{s_p-1}(\Rm^3)$ the critical Sobolev space of (CP1).  This is because the $\dot{H}^{s_p} \times \dot{H}^{s_p-1}$ norm is preserved if we apply the following natural rescaling transformation on a solution $u$ to (CP1). Given a solution $u$ and a positive constant $\lambda$, one can check that the function $u_\lambda = \lambda^{-2/(p-1)} u(x/\lambda,t/\lambda)$ is another solution to (CP1) with
\begin{align*}
\left\|(u_\lambda(\cdot,\lambda t'), \partial_t u_\lambda (\cdot,\lambda t'))\right\|_{\dot{H}^{s_p} \times \dot{H}^{s_p-1}} = \left\|(u(\cdot,t'), \partial_t u(\cdot,t'))\right\|_{\dot{H}^{s_p} \times \dot{H}^{s_p-1}}.
\end{align*}

\paragraph{Local theory} The main tool is the Strichartz estimate. An almost complete version of Strichartz estimates for 3D wave equation can be found in Ginibre-Velo \cite{strichartz}. A combination of suitable Strichartz estimates with a regular fixed-point argument gives the local well-posedness of this problem for initial data in the critical Sobolev space or energy space. Please see Kapitanski \cite{loc1} and Lindblad-Sogge \cite{ls}, for example, for more details of local theory. 

\paragraph{Energy conservation law} The energy is the most important conserved quantity of this equation. Throughout this work we use the notation $E$ for the energy. 
\[
 E(u, u_t) = \int_{\Rm^3} \left(\frac{1}{2}|\nabla u(x, t)|^2 +\frac{1}{2}|u_t(x, t)|^2 + \frac{1}{p+1}|u(x,t)|^{p+1}\right)\,dx = \hbox{Const}.
\]
In the energy sub-critical case $3\leq p<5$, any solution with a finite energy must exist for all time $t \in \Rm$. 

\paragraph{Global behaviour} In early 1990's M. Grillakis \cite{mg1, mg2} proved that any solution with initial data in the critical space $\dot{H}^1 \times L^2(\Rm^3)$ must scatter in both two time directions in the energy critical case $p=5$. By scattering we mean that a solution of nonlinear equation gets closer and closer to a solution of linear equation as $t$ goes to infinity. We expect that a similar result also holds for other exponents $p$. 
\begin{conjecture} \label{conjecture1}
 Any solution to (CP1) with initial data $(u_0,u_1) \in \dot{H}^{s_p} \times \dot{H}^{s_p-1}$ must exist for all time $t \in \Rm$ and scatter in both two time directions.
\end{conjecture}
\noindent This is still an open problem. Although there are many related results by different methods.

\paragraph{Scattering with a priori estimates} There are many works proving that if a solution $u$ with a maximal lifespan $I$ satisfies an a priori estimate
\begin{equation}
  \sup_{t \in I} \left\|(u(\cdot,t), u_t(\cdot, t))\right\|_{\dot{H}^{s_p} \times \dot{H}^{s_p-1} (\Rm^3)} < +\infty, \label{uniform UB}
 \end{equation}
then $u$ must exist globally in time and scatter. All these works use a compact-rigidity argument, which was introduced by Kenig-Merle in \cite{kenig, kenig1} for the study of energy critical wave and Sch\"{o}dinger equations. In the energy supercritical case $p>5$, one may refer to Duyckaerts et al. \cite{dkm2}, Kenig-Merle \cite{km}, Killip-Visan \cite{kv3} for radial solutions and Killip-Visan \cite{kv2} for non-radial solutions. In the energy subcritical case $p<5$, please see Dodson-Lawrie \cite{cubic3dwave} for $1+\sqrt{2}<p\leq 3$ and Shen \cite{shen2} for $3<p<5$, both in the radial case, as well as Dodson et al. \cite{nonradial3p5} for $3<p<5$ in the non-radial case. Please note that the argument in the papers mentioned above works not only in the defocusing case but also in the focusing case. In the energy critical case $p=5$, however, we have
\begin{itemize}
 \item In the defocusing case, the assumption \eqref{uniform UB} automatically holds by the energy conservation law. 
 \item In the focusing case, there exist solutions to (CP1) which satisfy the assumption \eqref{uniform UB} but fail to scatter. We call these kind of solutions type II blow-up solutions. One specific example of type II blow-up solutions is the ground state $W(x) = (1+|x|/3)^{-1/2}$. To learn more about global behaviour of type II blow-up solutions, please refer to Duyckaerts-Kenig-Merle \cite{dkmradial, se} in the radial case, and Duychaerts-Jia-Kenig \cite{djknonradial} in the non-radial case.
\end{itemize}

\paragraph{Non-radial initial data} The scattering of solutions can also be proved if the initial data satisfy stronger regularity and/or decay conditions. We start by results without a radial assumption.

\begin{itemize}
 \item Let $3 \leq p<5$. We may use the conformal conservation law method to prove the scattering of solutions if initial data satisfy 
 \[
  \int_{\Rm^3} \left[(|x|^2+1) (|\nabla u_0 (x)|^2 + |u_1(x)|^2) + |u_0(x)|^2 \right] dx < \infty.
 \]
 Please see Ginibre-Velo \cite{conformal2} and Hidano \cite{conformal} for more details. 
 \item Yang in his recent work \cite{yang1} considers energy momentum tensor and its associated current in order to obtain a uniform weighted energy bound and inverse polynomial decay of the energy flux through certain hyper surfaces. As an application scattering of solutions are proved under the assumption 
 \[
  \int_{\Rm^3} (1+|x|^\gamma)\left(\frac{1}{2}|\nabla u_0(x)|^2 + \frac{1}{2}|u_1(x)|^2 + \frac{1}{p+1}|u_0(x)|^{p+1}\right) dx < + \infty.
 \]
 Here the exponent $p$ and coefficient $\gamma$ satisfy
 \begin{align*}
  &\frac{1+\sqrt{17}}{2} < p < 5;& &\max\left\{\frac{5-p}{p-1}, 1\right\} < \gamma < \min\{p-1,2\}.&
 \end{align*}
\end{itemize}

\paragraph{Radial initial data} Radial assumption enables us to obtain the scattering of solutions under much weaker regularity/decay assumptions.
\begin{itemize}
 \item Dodson gives a proof of Conjecture \ref{conjecture1} in the radial case for $3 \leq p<5$ in his recent works \cite{claim1,claim2}. The radial assumption is essential, because the argument uses not only radial Strichartz estimates but also a conformal transformation for 3D wave equation, which was introduced in \cite{shen3} and works only for radial solutions. 
 \item The author introduces an inward/outward energy method in a recent work \cite{shenenergy}. As an application we may prove the scattering of solutions if the energy of initial data decays at a certain rate as $|x| \rightarrow +\infty$. More precisely, we assume
 \[
  \int_{\Rm^3} (1+|x|^\kappa)\left(\frac{1}{2}|\nabla u_0(x)|^2 + \frac{1}{2}|u_1(x)|^2 + \frac{1}{p+1}|u_0(x)|^{p+1}\right) dx < + \infty.
 \]
 Here $\kappa > \frac{5-p}{p+1}$ is a constant. In a more recent work \cite{sheninward} the author proves the scattering in the positive time direction by assuming that the inward energy of initial data decays at the same rate as above, regardless of the size and decay rate of outward energy. Please note that the decay assumption here is so weak that it can not guarantee $(u_0,u_1) \in \dot{H}^{s_p} \times \dot{H}^{s_p-1}$. As a result this inward/outward energy method discovers a scattering phenomenon which has not been covered by previously known scattering theory. The author would like to mention that the idea of inward/outward energy method partially coincides with the channel of energy method. To learn more about the channel of energy method, please refer to  Duyckaerts et al. \cite{tkm1}, Kenig et al. \cite{channel5d, channel}.
\end{itemize} 

\subsection{Motivation and Idea}

Since the method of inward/outward energy is a powerful tool to understand the asymptotic behaviour of radial solutions to 3D defocusing wave equation, as shown in the author's recent works \cite{shenenergy, sheninward}, we generalize this method to non-radial solutions in this work. 

\paragraph{Inward and outward energies} Before we define inward/outward energy of non-radial solutions, we first introduce a few notations.
 \begin{definition} For convenience we first define a few differential operators
 \begin{align*}
  (\mathbf{L} u)(x,t) & = \frac{1}{r}\partial_r(ru) = \nabla u(x,t)\cdot \frac{x}{|x|} + \frac{u(x,t)}{|x|};\\
  (\mathbf{L}_\pm u)(x,t) & = \frac{1}{r} (\partial_r\pm \partial_t) (ru) = \nabla u(x,t)\cdot \frac{x}{|x|} + \frac{u(x,t)}{|x|} \pm u_t (x,t). 
 \end{align*}
 When we consider initial data to (CP1), we also use the notation
 \begin{align*}
   &(\mathbf{L} u_0)(x) = \nabla u_0(x) \cdot \frac{x}{|x|} + \frac{u_0(x)}{|x|}& &\left[\mathbf{L}_\pm (u_0,u_1)\right](x) = \nabla u_0(x) \cdot \frac{x}{|x|} + \frac{u_0(x)}{|x|} \pm u_1(x).&
 \end{align*}
 \end{definition}
 \noindent Now we can define
 \begin{definition} \label{energies}
 Given any $t$ we define inward energy $E_-$ and outward energy $E_+$
\begin{align*}
 E_-(t) = \int_{\Rm^3} \left[\frac{1}{4}|\mathbf{L}_+ u(x,t)|^2 + \frac{1}{4}|\slashed{\nabla} u(x,t)|^2 + \frac{1}{2(p+1)}|u(x,t|^{p+1} \right] dx;\\
 E_+(t) = \int_{\Rm^3} \left[\frac{1}{4}|\mathbf{L}_- u(x,t)|^2 + \frac{1}{4}|\slashed{\nabla} u(x,t)|^2 + \frac{1}{2(p+1)}|u(x,t|^{p+1} \right] dx.
\end{align*}
Here $\slashed{\nabla}$ means the covariant derivative on the sphere centred at the origin with a fixed radius $|x|$. We have $|\nabla_x u|^2 = |u_r|^2 + |\slashed{\nabla} u|^2$. Given a radial symmetric region $\Sigma\subset \Rm^3$, we can also consider the inward/outward energy in the region
\[
 E_\pm (t;\Sigma) = \int_{\Sigma} \left[\frac{1}{4}|\mathbf{L}_\mp u(x,t)|^2 + \frac{1}{4}|\slashed{\nabla} u(x,t)|^2 + \frac{1}{2(p+1)}|u(x,t|^{p+1} \right] dx
\] 
In particular, we define $E_{\pm} (t;r_1,r_2) = E_{\pm} (t;\{x\in \Rm^3: r_1<|x|<r_2\})$.
\end{definition}
\begin{remark} \label{E plus add minus}
 By Lemma \ref{identity of w u energy}, we have 
\begin{align*}
 E_+(t) + E_-(t) & = \int_{\Rm^3} \left[\frac{1}{2}\left|\mathbf{L} u\right|^2 + \frac{1}{2}|u_t|^2 + \frac{1}{2}|\slashed{\nabla} u|^2 + \frac{1}{p+1}|u|^{p+1}\right] dx\\
 & = \int_{\Rm^3} \left[\frac{1}{2}|u_r|^2  + \frac{1}{2}|u_t|^2 + \frac{1}{2}|\slashed{\nabla} u|^2 + \frac{1}{p+1}|u|^{p+1}\right] dx = E.
\end{align*}
But in general we have
\[
 E_+(t; r_1,r_2) + E_-(t;r_1,r_2) \neq \int_{r_1<|x|<r_2}  \left[\frac{1}{2}|\nabla u(x,t)|^2  + \frac{1}{2}|u_t(x,t)|^2 +  \frac{1}{p+1}|u(x,t)|^{p+1}\right] dx.
\]
\end{remark}
\paragraph{Energy flux formula} As in the radial case, we then give an energy flux formula for inward/outward energy of non-radial solutions. Again the Morawetz estimates are essential to the proof. The major challenge in the non-radial case is to deal with the last two terms in the spherical coordinate version of the equation
\[
 (\partial_t-\partial_r)(\partial_t +\partial_r)(ru) = -r|u|^{p-1} u  + \frac{1}{r \sin \theta} \partial_\theta (u_\theta \sin \theta) + \frac{u_{\varphi \varphi}}{r \sin^2 \theta}.
\]
These two terms simply vanish in the radial case. We may deal with these terms containing second derivatives about $\theta, \varphi$ via integration by parts. In order to avoid the trouble of boundary terms we always consider spatially radial symmetric regions in the energy flux formula. 
The energy flux formula of inward/outward energy plays two important roles in our argument
\begin{itemize}
 \item It helps to give information about space-time distribution of the inward/outward energy, which gives plentiful information about the asymptotic behaviour of $u$. 
 \item It provides a framework that we can work on to obtain the weighted Morawetz estimates. 
\end{itemize}

\paragraph{Weighted Morawetz} If the energy of initial data decays at a certain rate when $|x| \rightarrow +\infty$, i.e. we have
\[
 \int_{\Rm^3} (1+|x|^\kappa) \left[\frac{1}{2}|\nabla u_0|^2 + \frac{1}{2}|u_1|^2 + \frac{1}{p+1}|u_0|^{p+1}\right] dx < + \infty,\quad 0<\kappa<1;
\]
then we may obtain the following weighted Morawetz 
\begin{align*}
 \int_{-\infty}^\infty \int_{\Rm^3} \frac{(|x|+t)^\kappa \left(|u(x,t)|^{p+1}+|\slashed{\nabla} u(x,t)|^2\right)}{|x|} dxdt < + \infty.
\end{align*}
As an application we have the decay estimates $E_-(t) \leq C t^{-\kappa}$ when $t>0$ is large.

\paragraph{Application on Scattering Theory} If $\kappa>\frac{5-p}{2}$, then the decay estimate $E_-(t) \leq C t^{-\kappa}$ implies $\|u\|_{L^q L^{p+1}(\Rm^+ \times \Rm^3)} < +\infty$ for $q$ slightly smaller than $2(p+1)/(5-p)$. We then combine this global estimate with the local theory and the energy conservation law to prove the scattering result.

\subsection{Main Results}

In this subsection we give three main theorems. Theorem \ref{main 1} gives the spatial energy distribution property of finite-energy solutions to (CP1) as $t\rightarrow \pm\infty$. This theorem is proved by an inward/outward energy method. As an application we may prove a scattering theory about solutions to (CP1) as given in Theorem \ref{main 2} and Theorem \ref{main 3}. Our assumptions are weaker than previously known scattering theory of non-radial solutions mentioned above.
\begin{theorem} \label{main 1}
Assume $3\leq p\leq 5$. Let $u$ be a solution to (CP1) with a finite energy $E$. Then we have the following asymptotic behaviour regarding the energy of $u$
\begin{align*}
& \lim_{t \rightarrow \pm \infty} \int_{\Rm^3} \left(\frac{1}{2}\left|\mathbf{L}_\pm u(x,t)\right|^2+ \frac{1}{2}|\slashed{\nabla} u(x,t)|^2 + \frac{1}{p+1}|u(x,t)|^{p+1} \right) dx = 0;\\
& \lim_{t \rightarrow \pm \infty} \int_{|x|<c|t|} \left(\frac{1}{2}|\nabla u(x,t)|^2 + \frac{1}{2}|u_t(x,t)|^2 + \frac{1}{p+1} |u(x,t)|^{p+1}\right) dx = 0, \quad c\in (0,1).
\end{align*}
\end{theorem}

\begin{remark}
 The first limit in the theorem above is equivalent to saying $\displaystyle \lim_{t \rightarrow \pm \infty} E_\mp (t) = 0$. 
\end{remark}

\begin{theorem} \label{main 2}
Assume $3< p< 5$ and $\kappa > \frac{5-p}{2}$. If initial data $(u_0,u_1)$ satisfy
\[
 E_{\kappa} (u_0,u_1) \doteq \int_{\Rm^3} (1+|x|^\kappa)\left(\frac{1}{2}|\nabla u_0|^2 + \frac{1}{2}|u_1|^2 + \frac{1}{p+1} |u_0|^{p+1}\right) dx < +\infty,
\]
 then the corresponding solution $u$ to (CP1) with initial data $(u_0,u_1)$ must scatter in both two time directions. More precisely, there exists $(v_0^\pm ,v_1^\pm) \in (\dot{H}^1\cap \dot{H}^{s_p}) \times (L^2\cap \dot{H}^{s_p-1})$, so that 
 \[
  \lim_{t \rightarrow \pm \infty} \left\|\begin{pmatrix} u(\cdot,t)\\ \partial_t u(\cdot,t)\end{pmatrix} - 
  \mathbf{S}_L (t)\begin{pmatrix}v_0^\pm \\ v_1^\pm\end{pmatrix}\right\|_{\dot{H}^s \times \dot{H}^{s-1}(\Rm^3)} = 0, \; \forall s\in [s_p,1].
 \]
 Here $\mathbf{S}_L (t)$ is the linear wave propagation operator. 
\end{theorem}
\begin{remark}
Initial data $(u_0,u_1)$ in Theorem \ref{main 2} also satisfy $ (u_0,u_1) \in \dot{H}^{s_p}\times \dot{H}^{s_p-1}(\Rm^3)$ by the Sobolev embedding. We put the details in Lemma \ref{initial data in theorem 2}. 
\end{remark}

\begin{remark}
Let $p$ and $\kappa$ be as in Theorem \ref{main 2}. We can prove the scattering of solution in the space $\dot{H}^1 \times L^2$ as $t \rightarrow +\infty$ as long as the total energy is finite and the inward energy satisfies 
 \[
  \int_{\Rm^3} (1+|x|^\kappa)\left(\frac{1}{4}\left|\mathbf{L}_+(u_0,u_1)\right|^2 +\frac{1}{4}|\slashed{\nabla} u_0|^2 + \frac{1}{2(p+1)} |u_0|^{p+1}\right) dx < +\infty,
 \]
 regardless of the size and decay rate of the outward energy. The idea comes form the author's work \cite{sheninward}:  the weighted Morawetz estimates for positive times depend on the inward energy of initial data only. This scattering result is in fact the major and key step to prove Theorem \ref{main 2}. Please see Section \ref{sec: proof main 2}.
\end{remark}

\begin{theorem} \label{main 3}
 Let $p=3$. If initial data $(u_0,u_1) \in \dot{H}^{1/2} \times \dot{H}^{-1/2}(\Rm^3)$ satisfy
 \[
   E_{1,0}(u_0,u_1) \doteq \int_{\Rm^3}  |x| \left(\frac{1}{2}|\nabla u_0(x)|^2 + \frac{1}{2}|u_1(x)|^2 + \frac{1}{4} |u_0(x)|^{4}\right) dx < +\infty,
\]
 then the corresponding solution $u$ to (CP1) with initial data $(u_0,u_1)$ must exist globally in time and scatter in both two time directions. More precisely, there exists $(v_0^\pm ,v_1^\pm) \in \dot{H}^{1/2} \times \dot{H}^{-1/2}$, so that 
 \[
  \lim_{t \rightarrow \pm \infty} \left\|\begin{pmatrix} u(\cdot,t)\\ \partial_t u(\cdot,t)\end{pmatrix} - 
  \mathbf{S}_L (t)\begin{pmatrix}v_0^\pm \\ v_1^\pm\end{pmatrix}\right\|_{\dot{H}^{1/2} \times \dot{H}^{-1/2}(\Rm^3)} = 0.
 \]
\end{theorem}

\begin{remark}
 Both $\dot{H}^{s_p}\times \dot{H}^{s_p-1}(\Rm^3)$ norm and the weighted energy $E_{0,1}$ defined in Theorem \ref{main 3} are invariant under the natural rescaling of (CP1).
\end{remark}

\begin{remark}
 These results for possibly non-radial initial data are weaker than our results for radial solutions, i.e. we have to make stronger assumptions on the decay rate of energy. This is because we lose many tools to investigate the asymptotic behaviour of solutions in the non-radial case. For example, we may rewrite the equation in the radial case in the form of 
\[
 (\partial_t-\partial_r)(\partial_t +\partial_r)(ru) = -r|u|^{p-1}u.
\] 
This immediately gives explicit estimates on variance of $(\partial_t \pm \partial_r)(ru)$ along characteristic lines $t\mp r = \hbox{Const}$. In the non-raidal case, however, we have 
\[
 (\partial_t-\partial_r)(\partial_t +\partial_r)(ru) = -r|u|^{p-1} r  + \frac{1}{r \sin \theta} \partial_\theta (u_\theta \sin \theta) + \frac{u_{\varphi \varphi}}{r \sin^2 \theta}.
\]
We are no longer able to analyze the variance and asymptotic behaviour of $(\partial_r\pm \partial_t)(ru)$ conveniently due to the presence of the derivatives about $\theta, \varphi$.  For another example, we do not have the pointwise estimate $|u(r,t)| \lesssim r^{-4/(p+3)}$ for radial solutions with a finite energy. This significantly undermines the effectiveness of (weighted) Morawetz estimates.
\end{remark}
 
 \subsection{The Structure of This Paper}
 
This paper is organized as follows. In section 2 we recall the Strichartz estimates, local theory and the Morawetz estimates, then give a few preliminary results. Next in Section 3 we give a general formula of inward and outward energy fluxes in the nonraidal case. Section 4 is divided to two parts. In the first part we give a few energy distribution properties of solutions by the energy flux formula. In the second part we prove the weighted Morawetz estimate. Finally in Section 5 we prove Theorem \ref{main 2} by combining the weighted Morawetz estimate with the local theory. 
\section{Preliminary Results}

We start by reminding the readers about the $\lesssim$ notation.

\paragraph{The $\lesssim$ symbol} We use the notation $A \lesssim B$ if there exists a constant $c$, so that the inequality $A \leq c B$ always holds.  In addition, a subscript of the symbol $\lesssim$ indicates that the constant $c$ is determined by the parameter(s) mentioned in the subscript but nothing else. In particular, $\lesssim_1$ means that the constant $c$ is an absolute constant. 


\subsection{Technical Lemmata}

\begin{lemma} \label{identity of w u energy}
Let $u \in \dot{H}^{1}(\Rm^3)$. Then $\displaystyle \int_{\Rm^3} \left|\mathbf{L} u\right|^2 dx = \int_{\Rm^3} |u_r|^2 dx$.
\end{lemma}
\begin{proof}
We first consider the integral over annulus $\{x\in \Rm^3: a<|x|<b\}$: 
\begin{align}
  \int_{a<|x|<b} \left|\mathbf{L} u\right|^2 dx 
   &=  \int_0^{2\pi} \int_0^{\pi} \int_a^b \left|u_r + \frac{u}{r}\right|^2 r^2 \sin \theta \,dr d\theta d\varphi \nonumber\\
  &= \int_0^{2\pi} \int_0^{\pi} \int_a^b \left[r^2 |u_r|^2 + \partial_r (ru^2)\right] \sin \theta\, dr d\theta d\varphi \nonumber\\
  &= \int_{a<|x|<b} |u_r|^2 dx + \frac{1}{b} \int_{|x|=b} |u|^2 d\sigma_b (x) - \frac{1}{a} \int_{|x|=a} |u|^2 d\sigma_a (x). \label{relationship of w and u}
\end{align}
Here $\sigma_r$ represents regular measure on sphere of radius $r$. Next we use Hardy's inequality and obtain 
\[
 \int_0^\infty \left(\frac{1}{r^2} \int_{|x|=r} |u(x)|^2 d\sigma_r(x)\right) dr = \int_{\Rm^3} \frac{|u(x)|^2}{|x|^2} dx \lesssim\|u\|_{\dot{H}^1}^2 < +\infty.
\]
As a result we have
\begin{align}
 &\liminf_{r \rightarrow 0^+} \frac{1}{r} \int_{|x|=r} |u(x)|^2 d\sigma_r(x) = 0;& &\liminf_{r \rightarrow +\infty} \frac{1}{r} \int_{|x|=r} |u(x)|^2 d\sigma_r(x) = 0;& \label{limit of two sides}
\end{align}
Finally we may make $a \rightarrow 0^+$ and $b \rightarrow +\infty$ in identity \eqref{relationship of w and u} with these two limits in mind to finish the proof.
\end{proof}

\begin{remark} \label{relationship of u w energy}
 If we use one limit at a time in identity \eqref{relationship of w and u}, we obtain two identities for any $\dot{H}^1(\Rm^3)$ function $u$ and any radius $R>0$
 \begin{align*}
  & \int_{|x|<R}  \left|\mathbf{L} u(x)\right|^2  dx = \int_{|x|<R} |u_r|^2 dx + \frac{1}{R} \int_{|x|=R} |u(x)|^2 d\sigma_R(x);\\
  & \int_{|x|>R}  \left|\mathbf{L} u(x)\right|^2  dx = \int_{|x|>R} |u_r|^2 dx -  \frac{1}{R} \int_{|x|=R} |u(x)|^2 d\sigma_R(x).
 \end{align*}
 This implies for any $\kappa > 0$ and $(u_0,u_1) \in \dot{H}^1 \times L^2$,
 \begin{align*}
   & \int_{\Rm^3} |x|^\kappa \left[\frac{1}{4}|\mathbf{L}_+ (u_0,u_1)|^2 + \frac{1}{4}|\mathbf{L}_- (u_0,u_1)|^2 + \frac{1}{2}|\slashed{\nabla} u_0|^2 + \frac{1}{p+1}|u_0|^{p+1}\right] dx \\
   = &  \int_{\Rm^3} |x|^{\kappa} \left[\frac{|\slashed{\nabla} u_0|^2}{2} + \frac{|u_1|^2}{2} + \frac{|u_0|^{p+1}}{p+1}\right] dx
   + \int_0^\infty \kappa R^{\kappa-1} \left( \int_{|x|>R} \frac{1}{2}\left|\mathbf{L} u_0(x)\right|^2 dx\right) dR\\
   \leq & \int_{\Rm^3} |x|^{\kappa} \left[\frac{|\slashed{\nabla} u_0|^2}{2} + \frac{|u_1|^2}{2} + \frac{|u_0|^{p+1}}{p+1}\right] dx
   +  \int_0^\infty \kappa R^{\kappa-1} \left( \int_{|x|>R} \frac{1}{2}|\partial_r u_0(x)|^2 dx\right) dR\\
   = & \int_{\Rm^3} |x|^\kappa \left[\frac{1}{2}|\nabla u_0|^2+ \frac{1}{2}|u_1|^2 + \frac{1}{p+1}|u_0|^{p+1}\right]  dx.
 \end{align*}
\end{remark}

\begin{lemma} \label{initial data in theorem 2}
Let $3\leq p<5$ and $\kappa>\frac{5-p}{2}$. If $(u_0,u_1) \in \dot{H}^1 \times L^2(\Rm^3)$ satisfies
\[
 E_{\kappa}^\ast (u_0,u_1) = \int_{\Rm^3} (|x|^\kappa+1)\left(|\nabla u_0|^2 + |u_1|^2\right) dx < +\infty.
\]
then we also have $(u_0,u_1) \in \dot{H}^{s_p} \times \dot{H}^{s_p-1} (\Rm^3)$.
\end{lemma}
\begin{proof}
By the Sobolev embedding $\dot{W}^{1,\frac{3(p-1)}{p+1}} \times L^\frac{3(p-1)}{p+1} \hookrightarrow \dot{H}^{s_p} \times \dot{H}^{s_p-1}$, it suffices to show  $(u_0,u_1) \in \dot{W}^{1,\frac{3(p-1)}{p+1}} \times L^\frac{3(p-1)}{p+1}$. This immediately follows H\"{o}lder's inequality
 \begin{align*}
  & \int_{\Rm^3} \left(|\nabla u_0|^\frac{3(p-1)}{p+1} + |u_1|^\frac{3(p-1)}{p+1}\right) dx \\
  \lesssim & \left[\int_{\Rm^3} (|x|+1)^\kappa \left(|\nabla u_0|^2 + |u_1|^2\right) dx\right]^{\frac{3(p-1)}{2(p+1)}} \left[\int_{\Rm^3} (|x|+1)^{-\frac{3(p-1)\kappa}{5-p}}dx \right]^{\frac{5-p}{2(p+1)}}\\
  \lesssim & \left(E_{\kappa}^\ast(u_0,u_1)\right)^{\frac{3(p-1)}{2(p+1)}} < +\infty. 
 \end{align*}
 Here we have $\frac{3(p-1)\kappa}{5-p} > \frac{3(p-1)}{2}\geq 3$. 
\end{proof}
\begin{lemma} \label{cutoff lemma}
Fix $s \in [-1,1]$. Let $\phi: \Rm^3 \rightarrow [0,1]$ be a fixed radial, smooth cut-off function satisfying 
\[
  \phi(x) = \left\{\begin{array}{ll}1, & \hbox{if}\; |x|\geq 1;\\ 0, & \hbox{if}\; |x|<1/2. \end{array}\right.
 \] 
 Then the operators $\{\mathbf{P}_r\}_{r\in \Rm^+}$ defined by $(\mathbf{P}_r f)(x) = \phi(x/r)\cdot f(x)$ are uniformly bounded from $\dot{H}^s(\Rm^3)$ to itself. 
 \[
  \left\|\mathbf{P}_r f\right\|_{\dot{H}^s(\Rm^3)} \lesssim_s \|f\|_{\dot{H}^s(\Rm^3)}.
 \]
 In addition, given any $f \in \dot{H}^s(\Rm^3)$, we have 
 \begin{align*}
  &\lim_{r \rightarrow 0^+} \|\mathbf{P}_r f -f\|_{\dot{H}^s(\Rm^3)} = 0;& &\lim_{r \rightarrow +\infty} \|\mathbf{P}_r f\|_{\dot{H}^s(\Rm^3)} = 0.&
 \end{align*}
\end{lemma}
\begin{proof}
 First of all, we can verify that $\|\mathbf{P}_r f\|_{\dot{H}^1} \lesssim_1 \|f\|_{\dot{H}^1}$ by a direct calculation 
 \[
  \int_{\Rm^3} \left|\nabla (\phi(x/r) f(x))\right|^2 dx \lesssim_1 \int_{\Rm^3} \left(|\nabla f(x)|^2 + \frac{|f(x)|^2}{|x|^2}\right) dx \lesssim_1 \|f\|_{\dot{H}^1}^2. 
 \]
 Here we apply Hardy's inequality. It is trivial that $\|\mathbf{P}_r f\|_{L^2} \leq \|f\|_{L^2}$. Therefore an interpolation immediately gives the uniform boundedness of $\mathbf{P}_r$ from $\dot{H}^s$ to itself for any $s \in [0,1]$. By duality this uniform boundedness is true for $s \in [-1,0]$ as well. Next let us prove the two limits as $r \rightarrow 0^+$ and $r \rightarrow \infty$. If $f$ is in the Schwartz class, then it is clear that the limits hold. In the general case we recall the fact that the Schwartz class is dense in the space $\dot{H}^s$ and apply the standard approximation techniques. Here we need to use the uniform boundedness of the operators $\mathbf{P}_r$. 
\end{proof}

\subsection{Strichartz estimates and Local Theory}

\paragraph{Strichartz estimates} The following Strichartz estimates on solutions to the linear wave equation play a key role in the local well-posedness  theory of nonlinear wave equations. Please see Proposition 3.1 in Ginibre-Velo \cite{strichartz}. Here we use the Sobolev version in dimension 3.
\begin{proposition} [Generalized Strichartz estimates] Let $2\leq q_1,q_2 \leq \infty$, $2\leq r_1,r_2 < \infty$ and $\rho_1,\rho_2,s\in \Rm$ be constants with 
\begin{align*}
 &1/q_i+1/r_i \leq 1/2, \; i=1,2; & &1/q_1+3/r_1=3/2-s+\rho_1;& &1/q_2+3/r_2=1/2+s+\rho_2.&
\end{align*}
Assume that $u$ is the solution to the linear wave equation
\[
 \left\{\begin{array}{ll} \partial_t u - \Delta u = F(x,t), & (x,t) \in \Rm^3 \times [0,T];\\
 u|_{t=0} = u_0 \in \dot{H}^s; & \\
 \partial_t u|_{t=0} = u_1 \in \dot{H}^{s-1}. &
 \end{array}\right.
\]
Then we have
\begin{align*}
 \left\|\left(u(\cdot,T), \partial_t u(\cdot,T)\right)\right\|_{\dot{H}^s \times \dot{H}^{s-1}} & +\|D_x^{\rho_1} u\|_{L^{q_1} L^{r_1}([0,T]\times \Rm^3)} \\
 & \leq C\left(\left\|(u_0,u_1)\right\|_{\dot{H}^s \times \dot{H}^{s-1}} + \left\|D_x^{-\rho_2} F(x,t) \right\|_{L^{\bar{q}_2} L^{\bar{r}_2} ([0,T]\times \Rm^3)}\right).
\end{align*}
Here the coefficients $\bar{q}_2$ and $\bar{r}_2$ satisfy $1/q_2 + 1/\bar{q}_2 = 1$, $1/r_2 + 1/\bar{r}_2 = 1$. The constant $C$ does not depend on $T$ or $u$. 
\end{proposition}

\paragraph{Chain rule} We also need the following ``chain rule'' for fractional derivatives. Please refer to Christ-Weinstein \cite{fchain}, Kenig et al. \cite{fchain2}, Staffilani \cite{fchain3} and Taylor \cite{fchain4} for more details. 
\begin{lemma} \label{chain rule}
 Assume a function $F$ satisfies $F(0) = F'(0) = 0$ and 
 \begin{align*}
  &|F'(a+b)| \leq C(|F'(a)| + |F'(b)|),& &|F''(a+b)| \leq C(|F''(a)| + |F''(b)|),&
 \end{align*}
 for all $a,b\in \Rm$. Then we have
 \[
  \|D^\alpha F(u)\|_{L^p (\Rm^3)} \leq C \|D^\alpha u\|_{L^{p_1} (\Rm^3)} \|F'(u)\|_{L^{p_2} (\Rm^3)}
 \]
 for $0<\alpha<1$ and $1/p = 1/p_1+1/p_2$, $1<p_1,p_2<\infty$.
\end{lemma}

\paragraph{Local theory} The local theory is a consequence of the Strichartz estimates and a fixed-point argument. Here we only give a few results that will be used later in this work. Please see Kapitanski \cite{loc1} and Lindblad-Sogge \cite{ls}, for instance, for more results and details about the local theory. We start by a few results with initial data in the critical Sobolev spaces.

\begin{proposition}[Existence and scattering criterion] \label{local existence}
For any initial data $(u_0,u_1) \in \dot{H}^{s_p} \times \dot{H}^{s_p-1}$, there exists a unique solution $u$ to (CP1) with a maximal lifespan $(-T_-,T_+)$ so that $(u(\cdot,t), \partial_t u(\cdot,t)) \in C((-T_-,T_+); \dot{H}^{s_p}\times \dot{H}^{s_p-1})$ and 
\begin{align*}
  \|u\|_{L^{2(p-1)}L^{2(p-1)}([a,b] \times \Rm^3)}<+\infty ,\; -T_-<a<b<T_+.
\end{align*}
In particular, if $\|u\|_{L^{2(p-1)}L^{2(p-1)}([0,T_+) \times \Rm^3)}<+\infty$, then $T_+=\infty$ and the solution $u$ scatters\footnote{When we mention scattering, we always assume the scattering happens in the critical Sobolev space $\dot{H}^{s_p} \times \dot{H}^{s_p-1}$ unless other space is specified.} in the positive time direction.
\end{proposition}

\begin{proposition}[Scattering with small initial data] \label{scattering with small initial data}
 There exists a constant $\delta = \delta(p)>0$, so that if the initial data satisfy $\|(u_0,u_1)\|_{\dot{H}^{s_p} \times \dot{H}^{s_p-1}} < \delta$, then the corresponding solution $u$ to (CP1) exists globally in time and scatters with $\|u\|_{L^{2(p-1)} L^{2(p-1)}(\Rm \times \Rm^3)} < + \infty$. 
\end{proposition}

\begin{corollary} \label{tail estimate}
 If $u$ is a solution to (CP1) defined in the time interval $(-T_-,T_+)$ with initial data $(u_0,u_1) \in \dot{H}^{s_p} \times \dot{H}^{s_p-1}$, then there exists a large radius $R$, so that
\[
 \int_{-T_-}^{T_+} \int_{|x|>R+|t|} |u(x,t)|^{2(p-1)} dxdt < +\infty.
\]
\end{corollary}
\begin{proof}
 Let us sketch out the proof. We recall the smooth cut-off operator $\mathbf{P}_r$ defined in Lemma \ref{cutoff lemma}. When $R$ is sufficiently large, we always have 
 \[
  \left\|(\mathbf{P}_R u_0, \mathbf{P}_R u_1)\right\|_{\dot{H}^{s_p} \times \dot{H}^{s_p-1}} < \delta.
 \]
 Here $\delta$ is the constant in Proposition \ref{scattering with small initial data}. We fix such a large radius $R$ and apply Proposition \ref{scattering with small initial data} to obtain that the corresponding solution $v$ to (CP1) with initial data $(\mathbf{P}_R u_0, \mathbf{P}_R u_1)$ satisfies 
 \[
  \int_{-\infty}^\infty \int_{\Rm^3} |v(x,t)|^{2(p-1)} dx dt < + \infty.
 \]
Since the initial data of $u$ and $v$ are exactly the same in the region $\{x\in \Rm^3: |x|>R\}$, we immediately finishes the proof by finite speed of propagation of wave equation.
\end{proof}

\noindent We also have continuous dependence of solutions on initial data. The following result is a direct consequence of the long time perturbation theory. Please see Lemma 2.5 of \cite{cubic3dwave} and Theorem 2.12 of \cite{shen2}, for example.
\begin{proposition} [Continuous dependence on initial data] \label{continuous dependence}
 Let $u$ be a solution to (CP1) with initial data $(u_0,u_1) \in \dot{H}^{s_p} \times \dot{H}^{s_p-1}$ and a maximal lifespan $I$. If $(u_{0,n},u_{1,n})$ is a sequence of initial data satisfying 
 \[
  \lim_{n \rightarrow +\infty} \|(u_{0,n},u_{1,n})-(u_0,u_1)\|_{\dot{H}^{s_p}\times \dot{H}^{s_p-1}} = 0,
 \]
 then the corresponding solutions $u_n$ to (CP1) with initial data $(u_{0,n}, u_{1,n})$ satisfy 
 \[
  \lim_{n\rightarrow +\infty} \sup_{t \in J} \|(u_n(\cdot,t),\partial_t u_n(\cdot,t))-(u(\cdot,t),\partial_t u(\cdot,t))\|_{\dot{H}^{s_p}\times \dot{H}^{s_p-1}} = 0,
 \]
 for any fixed compact subinterval $J \subset I$.
\end{proposition}

\noindent We can also consider local theory of energy subcritical wave equation with initial data in the energy space. 

\begin{lemma}[See Lemma 4.2 of \cite{sheninward}]\label{lemma local in energy space}
Assume $3\leq p<5$. Let $(u_0,u_1) \in \dot{H}^1 \times L^2(\Rm^3)$ be initial data. Then the Cauchy problem (CP1) has a unique solution $u$ in the time interval $[0,T]$ with 
 \begin{align*}
   &(u,u_t) \in C([0,T];\dot{H}^1 \times L^2(\Rm^3));& &u \in L^{\frac{2p}{p-3}}L^{2p} ([0,T]\times \Rm^3).&
 \end{align*}
 Here the minimal time length of existence $T = C_p \|(u_0,u_1)\|_{\dot{H}^1\times L^2}^{\frac{-2(p-1)}{5-p}}$.
\end{lemma}

\noindent Now let us consider a solution $u$ to (CP1) with a finite energy. The energy conservation law implies that the norm $\|(u(\cdot,t), u_t(\cdot,t))\|_{\dot{H}^1 \times L^2} \lesssim E^{1/2}$ is uniformly bounded for all time $t$ in the maximal lifespan of $u$. According to Lemma \ref{lemma local in energy space}, there exists a constant $T=T(p,E)$, so that if $u$ is still defined at time $t$, then $u$ is also defined for all time in $[t,t+T]$. It immediately follows that any solution to (CP1) with a finite energy exists globally in time.
\begin{proposition} \label{global existence finite energy}
Let $3\leq p<5$. If $u$ is a solution to (CP1) with a finite energy, then $u$ is defined for all time $t \in \Rm$. 
\end{proposition}

\paragraph{Scattering in different spaces} Finally we give a technical lemma about scattering in different spaces. 
\begin{lemma} \label{scattering in different spaces} 
Assume that a solution to (CP1) scatters in two different levels of Sobolev spaces ($1/2 \leq s_1<s_2 \leq 1$)
\[
 \lim_{t \rightarrow + \infty} \left\|\begin{pmatrix} u(\cdot,t)\\ \partial_t u(\cdot,t)\end{pmatrix} - 
  \mathbf{S}_L (t)\begin{pmatrix}v_0^{(i)} \\ v_1^{(i)}\end{pmatrix}\right\|_{\dot{H}^{s_i} \times \dot{H}^{s_i-1}(\Rm^3)} = 0, \; i=1,2.
\]
Then we always have $(v_0^{(1)}, v_1^{(1)}) = (v_0^{(2)}, v_1^{(2)})$ and 
\[
 \lim_{t \rightarrow + \infty} \left\|\begin{pmatrix} u(\cdot,t)\\ \partial_t u(\cdot,t)\end{pmatrix} - 
  \mathbf{S}_L (t)\begin{pmatrix}v_0^{(1)} \\ v_1^{(1)}\end{pmatrix}\right\|_{\dot{H}^{s} \times \dot{H}^{s-1}(\Rm^3)} = 0, \; s \in [s_1,s_2].
\]
\end{lemma}

\begin{proof}
 Since the operator $\mathbf{S}_L(t)$ preserves $\dot{H}^s \times \dot{H}^{s-1}$ norms, we have
 \[
 \lim_{t \rightarrow + \infty} \left\|\mathbf{S}_L (-t)\begin{pmatrix} u(\cdot,t)\\ \partial_t u(\cdot,t)\end{pmatrix} - 
  \begin{pmatrix}v_0^{(i)} \\ v_1^{(i)}\end{pmatrix}\right\|_{\dot{H}^{s_i} \times \dot{H}^{s_i-1}(\Rm^3)} = 0, \; i=1,2.
\]
This means that both $(v_0^{(1)}, v_1^{(1)})$ and $(v_0^{(2)}, v_1^{(2)})$ are the limit of $\mathbf{S}_L(-t) (u(\cdot,t),\partial_t u(\cdot,t))$ as $t\rightarrow +\infty$ in the sense of tempered distribution. Thus $(v_0^{(1)}, v_1^{(1)}) = (v_0^{(2)}, v_1^{(2)})$. The scattering of $u$ in the space $\dot{H}^s \times \dot{H}^{s-1}$ with $s \in (s_1,s_2)$ then follows an interpolation between $s_1$ and $s_2$. 
\end{proof}

\subsection{Morawetz Estimates} \label{sec:morawetz}
 
We first recall the classic Morawetz estimate for wave equation, as given in Perthame and Vega's work \cite{benoit}. Here we use the 3-dimensional case.
\begin{theorem} 
Let $u$ be a solution to (CP1) defined in a time interval $[0,T]$ with a finite energy $E$. Then we have the following inequality for any $R>0$. Here $\sigma_R$ is the regular surface measure of the sphere $|x|=R$. 
\begin{align}
 & \frac{1}{2R}\int_0^T \!\!\int_{|x|<R}(|\nabla u|^2+|u_t|^2) dx dt + \frac{1}{2R^2} \int_0^T \!\!\int_{|x|=R} |u|^2 d\sigma_R dt + \frac{p-2}{(p+1)R} \int_0^T \!\!\int_{|x|<R} |u|^{p+1} dx dt \nonumber \\
 & \quad + \frac{p-1}{p+1} \int_0^T \int_{|x|>R} \frac{|u|^{p+1}}{|x|} dx dt + \int_0^T \int_{|x|>R} \frac{|\slashed{\nabla} u|^2}{|x|} dx dt + \frac{1}{R^2} \int_{|x|<R} |u(x,T)|^2 dx \leq 2E. \label{morawetz1}
\end{align}
\end{theorem}
\begin{remark}
The notations $p$ and $E$ represent slightly different constants in the original paper \cite{benoit} and this current paper. Here we rewrite the inequality in the setting of the current work. The coefficient of the integral $\int_{B(0,R)} |u(T)|^2 dx$ was $\frac{d^2-1}{4R^2}$ (in the 3-dimensional case $\frac{2}{R^2}$) in the original paper. But the author believes that this is a minor typing mistake. It should have been $\frac{d^2-1}{8R^2}$ instead. 
\end{remark} 
\begin{remark}
 The left hand side of the original inequality given in Perthame and Vega's work does not contain the term $\int_0^T \int_{|x|>R} \frac{|\slashed{\nabla} u|^2}{|x|} dx dt$. Instead this term is simply discarded since it must be nonnegative. In particular this term vanishes if $u$ is a radial solution. But a careful review of the proof given in \cite{benoit} clearly shows that we may put this term in the left hand of the inequality as well. A similar inequality is also proved in Yang \cite{yang1}
 \[
  \int_{-\infty}^\infty \int_{\Rm^3} \frac{|u|^{p+1} + |\slashed{\nabla} u|^2}{|x|} dx dt \leq C E.
 \] 
\end{remark}

\paragraph{Global Integral Estimates} Let us recall that any finite-energy solution to (CP1) is defined for all time $t\in \Rm$. Since none of the coefficients in the Morawetz estimate \eqref{morawetz1} depend on the time $T$. We may substitute the upper limit $T$ of the integrals by $+\infty$, as long as we ignore the last term in the left hand side. We may also substitute the lower limit $0$ of the integrals by $-\infty$, thanks to the energy conservation law. For convenience we combine integral over the same region together and write
\begin{align}
 & \frac{1}{R}\int_{-\infty}^\infty \int_{|x|<R}\left(\frac{1}{2}|\nabla u|^2+\frac{1}{2}|u_t|^2 + \frac{p-2}{p+1}|u|^{p+1}\right) dx dt + \frac{1}{2R^2} \int_{-\infty}^\infty \int_{|x|=R} |u|^2 d\sigma_R(x) dt \nonumber \\
 & \quad + \int_{-\infty}^\infty \int_{|x|>R} \left(\frac{p-1}{p+1} \cdot \frac{|u|^{p+1}}{|x|} + \frac{|\slashed{\nabla} u|^2}{|x|}\right) dx dt \leq 2E. \label{morawetz2}
\end{align}
\noindent This immediately gives the following results. (In order to obtain the third inequality we let $R\rightarrow 0^+$.)
\begin{corollary} \label{cor u morawetz}
 Let $u$ be a solution to (CP1) with a finite energy $E$. Then $u$ satisfies ($R>0$)
 \begin{align*}
   \int_{-\infty}^\infty \int_{|x|<R}\left(|\nabla u|^2+ |u_t|^2 + |u|^{p+1}\right) dx dt & \lesssim_p RE;\\
    \int_{-\infty}^\infty \int_{|x|=R} |u(x,t)|^2 d\sigma_R(x) dt & \leq 2R^2 E; \\
   \int_{-\infty}^{+\infty} \int_{\Rm^3} \frac{|u(x,t)|^{p+1}+|\slashed{\nabla} u(x,t)|^2}{|x|} dx dt & \lesssim_{p} E.
 \end{align*}
\end{corollary}

\section{Energy Flux for Inward and Outward Energies} \label{sec: flux}

In this section we consider the inward and outward energies given in Definition \ref{energies} and give energy flux formula of them. We first give the statement of energy flux formula in the first subsection. The proof is put in the second subsection. 

\subsection{General Energy Flux Formula}

\paragraph{Region involved in this work} Throughout this work, when we apply energy flux formula we always consider a spatially radial symmetric region $\Omega \subset \Rm^3 \times \Rm$, so that it can be expressed by 
\[
 \Omega = \{(r\sin \theta \cos \varphi,r \sin \theta \sin \varphi, r \cos \theta,t)\in \Rm^3 \times \Rm: (r,t) \in \Phi, \theta \in [0,\pi], \varphi \in [0,2\pi]\},
\]
if we use spherical coordinates $(r,\theta,\varphi)$ in $\Rm^3$. Here $\Phi$ is a bounded, closed region in $\Rm_r^+ \times \Rm_t$, whose boundary $\partial \Phi$ is a simple curve consisting of finite number of line segments paralleled to either $t\pm r = 0$ or coordinate axes. Thus the boundary surface $\partial \Omega$ consists of finite pieces of annulus, circular cylinders and cones. When necessary we also allow a line segment of $t$-axis to be part of the boundary $\partial \Phi$. In this case part of boundary surface $\partial \Omega$ is degenerate and shrinks to a line segment. 

\begin{proposition}[General Energy Flux] \label{energy flux formula}
Assume that $3\leq p\leq 5$. Let $u$ be a solution to (CP1) with a finite energy $E$. We define 
 \begin{align*}
  \mathbf{V}_- = & \left[ -\frac{1}{4} \left|\mathbf{L}_+ u\right|^2 + \frac{1}{4}\left|\slashed{\nabla} u\right|^2 + \frac{|u|^{p+1}}{2(p+1)}\right]\frac{\vec{x}}{|x|} + \left[\frac{1}{4}\left|\mathbf{L}_+ u\right|^2 + \frac{1}{4}\left|\slashed{\nabla} u\right|^2 + \frac{|u|^{p+1}}{2(p+1)}\right]\vec{e};\\
   \mathbf{V}_+ = & \left[ +\frac{1}{4} \left|\mathbf{L}_- u\right|^2 - \frac{1}{4}\left|\slashed{\nabla} u\right|^2 - \frac{|u|^{p+1}}{2(p+1)}\right]\frac{\vec{x}}{|x|} + \left[\frac{1}{4}\left|\mathbf{L}_- u\right|^2 + \frac{1}{4}\left|\slashed{\nabla} u\right|^2 + \frac{|u|^{p+1}}{2(p+1)}\right]\vec{e};
 \end{align*} 
 Here $\vec{x}, \vec{e}$ represent vectors $(x_1,x_2,x_3,0), (0,0,0,1) \in \Rm^3 \times \Rm$, respectively. If $\Omega$ is a region in $\Rm^3 \times \Rm$ as described above, then we have
 \[
  \int_{\partial \Omega} \mathbf{V}_{\pm} \cdot d\mathbf{S} = \pm  \iint_\Omega \left(\frac{p-1}{2(p+1)}\frac{|u|^{p+1}}{|x|} + \frac{1}{2} \frac{|\slashed{\nabla} u|^2}{|x|} \right) dx dt.
 \]
 In addition, there exist a nonnegative, finite and continuous\footnote{By continuity we mean $\mu((-\infty,t])$ is a continuous function of $t$.} measure $\mu$ on $\Rm$ with $\mu(\Rm) \lesssim_p E$, which is determined by $u$ and independent to $\Omega$, so that if $\partial \Phi$ contains a line segment $[t_1,t_2]$ of the $t$-axis, then the identity above still holds if we add $\mp \pi \mu([t_1,t_2])$ to the left hand side accordingly.
\end{proposition}

\paragraph{Surface integrals} We first have a look at what the surface integrals look like for different types of boundary hyper-surfaces $\Sigma$, as shown in table \ref{surface integrals in energy flux}. Please note that the arrows in the first column indicates the orientation of the surface.
\begin{table}[h]
\caption{Surface integrals in energy flux formula}
\begin{center}
\begin{tabular}{|c|c|c|}\hline
 Boundary type & Inward Energy Case & Outward Energy Case\\
 \hline
 Horizontally $\uparrow$ & $\int_\Sigma \left(\frac{\left|\mathbf{L}_+ u\right|^2}{4} \!+\! \frac{|\slashed{\nabla} u|^2}{4} \!+\! \frac{|u|^{p+1}}{2(p+1)}\right) dS$ & 
 $\int_\Sigma \left(\frac{\left|\mathbf{L}_- u\right|^2}{4} \!+\! \frac{|\slashed{\nabla} u|^2}{4} \!+\! \frac{|u|^{p+1}}{2(p+1)}\right) dS$ \\
 \hline
 Horizontally $\downarrow$ & $-\int_\Sigma \left(\frac{\left|\mathbf{L}_+ u\right|^2}{4} \!+\! \frac{|\slashed{\nabla} u|^2}{4} \!+\! \frac{|u|^{p+1}}{2(p+1)}\right) dS$ & 
 $-\int_\Sigma \left(\frac{\left|\mathbf{L}_- u\right|^2}{4} \!+\! \frac{|\slashed{\nabla} u|^2}{4} \!+\! \frac{|u|^{p+1}}{2(p+1)}\right) dS$\\
 \hline
 $\!|x|\!=\!r_0$, Outward & $\int_\Sigma \left(-\frac{\left|\mathbf{L}_+ u\right|^2}{4} \!+\! \frac{|\slashed{\nabla} u|^2}{4} \!+\! \frac{|u|^{p+1}}{2(p+1)}\right) dS$ & 
 $\int_\Sigma \left(\frac{\left|\mathbf{L}_- u\right|^2}{4} \!-\! \frac{|\slashed{\nabla} u|^2}{4} \!-\! \frac{|u|^{p+1}}{2(p+1)}\right) dS$ \\
 \hline
$\!|x|\!=\!r_0$, Inward & $\int_\Sigma \left(\frac{\left|\mathbf{L}_+ u\right|^2}{4} \!-\! \frac{|\slashed{\nabla} u|^2}{4} \!-\! \frac{|u|^{p+1}}{2(p+1)}\right) dS$ & 
 $\int_\Sigma \left(-\frac{\left|\mathbf{L}_- u\right|^2}{4} \!+\! \frac{|\slashed{\nabla} u|^2}{4} \!+\! \frac{|u|^{p+1}}{2(p+1)}\right) dS$ \\
 \hline
 $|x|=0$ & $-\pi \mu([t_1,t_2])$ & $\pi \mu([t_1,t_2])$ \\
 \hline
 Backward Cone$\uparrow$ & $\frac{1}{\sqrt{2}} \int_\Sigma \left(\frac{|\slashed{\nabla} u|^2}{2} \!+\! \frac{|u|^{p+1}}{p+1}\right) dS$ & $\frac{1}{2\sqrt{2}}\int_\Sigma |\mathbf{L}_- u|^2 dS$\\
 \hline
 Backward Cone$\downarrow$ & $-\frac{1}{\sqrt{2}} \int_\Sigma \left(\frac{|\slashed{\nabla} u|^2}{2} \!+\! \frac{|u|^{p+1}}{p+1}\right) dS$ & $-\frac{1}{2\sqrt{2}}\int_\Sigma |\mathbf{L}_- u|^2 dS$\\
 \hline
 Light Cone $\uparrow$ & $\frac{1}{2\sqrt{2}}\int_\Sigma |\mathbf{L}_+ u|^2 dS$ & $\frac{1}{\sqrt{2}} \int_\Sigma \left(\frac{|\slashed{\nabla} u|^2}{2} \!+\! \frac{|u|^{p+1}}{p+1}\right) dS$\\
 \hline
 Light Cone $\downarrow$ & $-\frac{1}{2\sqrt{2}}\int_\Sigma |\mathbf{L}_+ u|^2 dS$ & $-\frac{1}{\sqrt{2}} \int_\Sigma \left(\frac{|\slashed{\nabla} u|^2}{2} \!+\! \frac{|u|^{p+1}}{p+1}\right) dS$\\
  \hline
\end{tabular}
\end{center}
\label{surface integrals in energy flux}
\end{table}

\begin{figure}[h]
 \centering
 \includegraphics[scale=1.0]{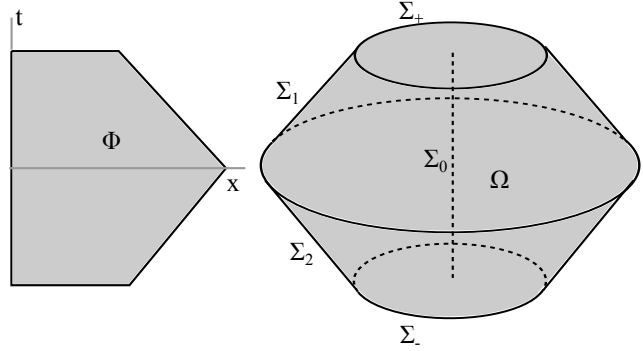}
 \caption{Illustration of the region $\Omega$} \label{figure example}
\end{figure}

\paragraph{Physical Interpretation} Now let us show how to use this energy flux formula. We show this by an example. The region $\Omega$ is defined by 
\[ 
 \Omega = \{(x,t): -1\leq t\leq 1, |x|+t\leq 2, t-|x|\geq -2\}.
\]
The boundary surface consists of five parts, $\Sigma_+$, $\Sigma_1$, $\Sigma_2$, $\Sigma_-$ and $\Sigma_0$, as shown in figure \ref{figure example}. The last one is degenerate. We then write down the energy flux formula for inward energy. 
\begin{align*}
 \int_{\Sigma_+} \mathbf{V}_-\cdot d\mathbf{S} + \int_{\Sigma_1} \mathbf{V}_-\cdot d\mathbf{S} & + \int_{\Sigma_2} \mathbf{V}_-\cdot d\mathbf{S}
 +\int_{\Sigma_-} \mathbf{V}_-\cdot d\mathbf{S} + \pi \mu([-1,1])\\
 &  = -\iint_\Omega \left(\frac{p-1}{2(p+1)}\frac{|u|^{p+1}}{|x|} + \frac{1}{2} \frac{|\slashed{\nabla} u|^2}{|x|} \right) dx dt.
\end{align*}
We move the surface integrals over $\Sigma_1$, $\Sigma_2$, the term $\pi \mu([-1,1])$ to the right hand side and then calculate the integrals in details.
\begin{align*}
 E_-(1;0,1) - E(-1;0,1) = - \frac{1}{\sqrt{2}} & \int_{\Sigma_1} \left(\frac{|\slashed{\nabla} u|^2}{2} \!+\! \frac{|u|^{p+1}}{p+1}\right) dS + \frac{1}{2\sqrt{2}}\int_{\Sigma_2} |\mathbf{L}_+ u|^2 dS \\
 & -\pi \mu([-1,1]) - \iint_\Omega \left(\frac{p-1}{2(p+1)}\frac{|u|^{p+1}}{|x|} + \frac{1}{2} \frac{|\slashed{\nabla} u|^2}{|x|} \right) dx dt.
\end{align*}
The left hand side is the difference of inward energy in the region $\Sigma_+$ and inward energy in the region $\Sigma_-$. This change is a sum of four terms, as shown in the right hand side.  The first term is the surface integral over $\Sigma_1$. This is a loss term. The integral of $|u|^{p+1}/(p+1)$ represents the amount of energy loss on $\Sigma_1$ due to non-linear effect. While the integral of $|\slashed{\nabla} u|^2/2$ represents the amount of energy loss on $\Sigma_1$ by linear propagation. The second term is a gain. It represents the amount of energy moving into the region $\Omega$ through the surface $\Sigma_2$ by linear propagation. The third term is again a loss. It represents the amount energy carried by inward waves which go through the origin thus change to outward waves during the time period $[-1,1]$. The last term is a double integral in space-time region $\Omega$. This is the amount of inward energy transformed to outward energy by both linear ($|\slashed{\nabla} u|^2$) or nonlinear ($|u|^{p+1}$) wave propagation. 

\begin{remark}
 The gain or loss represented by surface integral looks as if it happens on the boundary surface. However, they actually contain contributions from both boundary effect and wave propagation. For example, if we consider the change rate of inward energy contained inside the back forward light cone $\Sigma_1 = \{(x,t): |x|+t=2, 0\leq t\leq 1\}$ with respect to $t$, we may calculate the derivative
 \begin{align*}
  & \frac{d}{dt} \int_{|x|<2-t} \left(\frac{1}{4}\left|\mathbf{L}_+ u\right|^2 + \frac{1}{4}\left|\slashed{\nabla} u\right|^2 + \frac{|u|^{p+1}}{2(p+1)}\right) dx = I_1+I_2.
 \end{align*}
 Here we have
 \begin{align*}
 I_1 = & - \int_{|x|=2-t} \left(\frac{1}{4}\left|\mathbf{L}_+ u\right|^2 + \frac{1}{4}\left|\slashed{\nabla} u\right|^2 + \frac{|u|^{p+1}}{2(p+1)}\right) d\sigma_{2-t} (x);\\
 I_2 =  & \int_{|x|<2-t} \frac{d}{dt} \left(\frac{1}{4}\left|\mathbf{L}_+ u\right|^2 + \frac{1}{4}\left|\slashed{\nabla} u\right|^2 + \frac{|u|^{p+1}}{2(p+1)}\right) dx.
 \end{align*}
 The term $I_1$ is the contribution by the change of integral region as $t$ increases. While $I_2$ is the contribution by nonlinear wave propagation inside the cone. If we integrate $I_1$ to obtain the contribution of the boundary effect on the surface $\Sigma_1$, we have
 \begin{align*}
  \int_{0}^1 I_1(t) dt =  - \frac{1}{\sqrt{2}} \int_{\Sigma_1} \left(\frac{1}{4}\left|\mathbf{L}_+ u\right|^2 + \frac{1}{4}\left|\slashed{\nabla} u\right|^2 + \frac{|u|^{p+1}}{2(p+1)}\right) dS.
 \end{align*}
 This is a half of the energy flux in our formula, plus an extra term regarding $|\slashed{\nabla} u|^2$. The other half will come form the integral of $I_2$, i.e. the contribution of nonlinear wave propagation. The integral of $I_2$ will also come with a term regarding $|\slashed{\nabla} u|^2$ thus kill the extra term mentioned above. 
\end{remark}

\paragraph{Notation of Energy Fluxes on Cones} For convenience we define
\begin{definition}[Notations of Light Cones] 
Given $s,\tau \in \Rm$, we use the following notations for backward($-$) and forward($+$) light cones  
\begin{align*}
 &C^-(s) = \{(x,t): |x|+t = s\};& &C^+(\tau) = \{(x,t): t-|x|=\tau\}.&
\end{align*}
We also need to consider their truncated version, i.e. the cones between two given times $t_1, t_2$. 
\begin{align*}
 &C^-(s;t_1,t_2) = \{(x,t): |x|+t = s, t_1\leq t\leq t_2\};&\\
  &C^+(\tau; t_1,t_2) = \{(x,t): t-|x|=\tau, t_1\leq t\leq t_2\}.&
\end{align*}
\end{definition}
\begin{definition}[Notations of Energy fluxes]
 Given $s,\tau \in \Rm$, we define energy fluxes through light cones
\begin{align*} 
 Q_-^- (s) & = \frac{1}{\sqrt{2}} \int_{C^-(s)} \left(\frac{1}{p+1}|u|^{p+1} +  \frac{1}{2}\left|\slashed{\nabla} u\right|^2 \right) dS;\\
 Q_+^-(s) & = \frac{1}{2\sqrt{2}} \int_{C^-(s)} |\mathbf{L}_- u|^2 dS;\\
 Q_-^+(\tau) & = \frac{1}{2\sqrt{2}} \int_{C^+(\tau)} |\mathbf{L}_+ u|^2 dS;\\
 Q_+^+(\tau) & = \frac{1}{\sqrt{2}} \int_{C^+(\tau)} \left(\frac{1}{p+1}|u|^{p+1} +  \frac{1}{2}\left|\slashed{\nabla} u\right|^2 \right) dS.
\end{align*}
The upper index tells us whether this is energy flux through backward light cone($-$), or forward light cone($+$). The lower index indicates whether this is energy flux of inward energy ($-$) or outward energy ($+$). We can also consider their truncated version, which is the energy flux through the given light cone during the given time period.  
\begin{align*} 
 Q_-^- (s; t_1,t_2) & =  \frac{1}{\sqrt{2}} \int_{C^-(s;t_1,t_2)} \left(\frac{1}{p+1}|u|^{p+1} +  \frac{1}{2}\left|\slashed{\nabla} u\right|^2 \right) dS,& &t_1<t_2\leq s;\\
 Q_+^-(s; t_1,t_2) & = \frac{1}{2\sqrt{2}} \int_{C^-(s;t_1,t_2)} |\mathbf{L}_- u|^2 dS, & & t_1<t_2\leq s;\\
 Q_-^+(\tau; t_1,t_2) & = \frac{1}{2\sqrt{2}} \int_{C^+(\tau;t_1,t_2)} |\mathbf{L}_+ u|^2 dS,& &\tau\leq t_1<t_2;\\
 Q_+^+(\tau; t_1,t_2) & = \frac{1}{\sqrt{2}} \int_{C^+(\tau;t_1,t_2)} \left(\frac{1}{p+1}|u|^{p+1} +  \frac{1}{2}\left|\slashed{\nabla} u\right|^2 \right) dS,& &\tau\leq t_1<t_2.
\end{align*}
\end{definition}

\begin{remark} \label{bound of Q}
 The sums $Q_-^-(s)+Q_+^-(s)$ and $Q_-^+(\tau)+Q_+^+(\tau)$ are exactly fluxes of full energy across the light cones $C^-(s)$ and $C^+(\tau)$, respectively. In face, we have
 \begin{align*}
  Q_-^+(\tau)+Q_+^+(\tau) & = \frac{1}{\sqrt{2}} \int_{C^+(\tau)} \left(\frac{1}{p+1}|u|^{p+1} +  \frac{1}{2}\left|\slashed{\nabla} u\right|^2 + \frac{1}{2}\left|\mathbf{L}_+ u\right|^2\right) dS \\
  & = \int_{\Rm^3} \left(\frac{1}{p+1}|\tilde{u}|^{p+1} +  \frac{1}{2}\left|\slashed{\nabla} \tilde{u}\right|^2+\frac{1}{2}\left|\frac{1}{r}\partial_r (r\tilde{u})\right|^2 \right) dx \\
  & =  \int_{\Rm^3} \left(\frac{1}{p+1}|\tilde{u}|^{p+1} +  \frac{1}{2}\left|\slashed{\nabla} \tilde{u}\right|^2+\frac{1}{2}\left|\partial_r \tilde{u}\right|^2 \right) dx \\
  & = \frac{1}{\sqrt{2}} \int_{C^+(\tau)} \left(\frac{1}{p+1}|u|^{p+1} +  \frac{1}{2}\left|\slashed{\nabla} u\right|^2 + \frac{1}{2}\left|(\partial_r + \partial_t)u\right|^2\right) dS \leq E.
 \end{align*}
Here we have $\tilde{u}(x) = u(x,|x|+\tau) \in \dot{H}^1 \cap L^{p+1} (\Rm^3)$ by the well-know energy flux formula of full energy. We also need to appy Lemma \ref{identity of w u energy} in the calculation. The situation of $Q_-^-(s) + Q_+^-(s)$ is similar. As a result all the energy fluxes defined above are dominated by the full energy $E$.
\end{remark}

\paragraph{Notation of double integral} For convenience we also use the following notation for the double integral in the energy flux formula
\begin{definition}[Morawetz integral]
 Given any region $\Omega$ in $\Rm^3 \times \Rm$, we define the Morawetz integral 
 \[
  \mathcal{M}(\Omega) = \iint_\Omega \left(\frac{p-1}{2(p+1)}\cdot\frac{|u(x,t)|^{p+1}}{|x|} + \frac{1}{2} \cdot\frac{|\slashed{\nabla} u(x,t)|^2}{|x|} \right) dx dt.
 \]
\end{definition}

\subsection{Proof of Energy Flux Formula}

Now let us give a proof of Proposition \ref{energy flux formula}.  Without loss of generality let us prove the energy flux formula of inward energy. Then we may obtain the energy flux formula of outward energy by either of the following two ways
\begin{itemize}
 \item We may follow the same argument as in the inward energy case.
 \item We may consider the difference of the energy flux formula for full energy and inward energy. Please note that in general the sum of energy fluxes of inward and outward energies through a hypersurface is NOT the energy flux of the full energy through the same hypersurface. Similarly in general the sum of inward and outward energies in an annulus $\Sigma = \{x\in \Rm: R_1<|x|<R_2\}$ is NOT the energy contained in $\Sigma$, as shown in Remark \ref{E plus add minus}. In fact error terms as below may appear 
 \[
 \pm \frac{1}{2R}\int_{|x|=R} |u(x,t)|^2 d\sigma_R (x), 
 \]
 as we found in the proof of Lemma \ref{identity of w u energy}. We have to keep track of them and show that all these terms are finally canceled out.
\end{itemize}

\paragraph{The case away from $t$-axis}  If the region $\Phi$ is away from the $t$-axis, then we may apply Guass' formula on the surface integral. For convenience we define $w(x,t) = |x| u(x,t)$. We will also use spherical coordinates in order to simplify the calculation and take advantage of the spatial symmetric assumption on $\Omega$. Our argument below involves second derivatives. We may apply smooth approximation techniques if the solution is not sufficiently smooth. We start by recalling
\begin{equation} \label{expression of slashed}
 |\nabla u|^2 = |u_r|^2 + \frac{|u_\theta|^2}{r^2} + \frac{|u_\varphi|^2}{r^2 \sin^2 \theta}\quad \Rightarrow \quad |\slashed{\nabla} u|^2 = 
 \frac{1}{r^2}\left(|u_\theta|^2+\frac{|u_\varphi|^2}{\sin^2 \theta}\right).
\end{equation}
We may plug this in $\mathbf{V}_-$ and calculate its divergence
\begin{align*}
 4\;\hbox{div}\mathbf{V}_- = & \left(\partial_r + \frac{2}{r}\right)\left[\frac{1}{r^2}\left(\frac{2}{p+1}\cdot\frac{|w|^{p+1}}{r^{p-1}} - |w_r+w_t|^2 + |u_\theta|^2 + \frac{|u_\varphi|^2}{\sin^2 \theta}\right)\right]\\
 & \quad + \partial_t \left[\frac{1}{r^2}\left(\frac{2}{p+1}\cdot\frac{|w|^{p+1}}{r^{p-1}} + |w_r+w_t|^2 + |u_\theta|^2 + \frac{|u_\varphi|^2}{\sin^2 \theta}\right)\right]
\end{align*}
We observe the fact $\left(\partial_r + 2/r\right)(X/r^2) = (1/r^2)\partial_r X$ and obtain
\begin{align*}
 4\; \hbox{div} \mathbf{V}_- & = \frac{2}{r^2}(w_r+w_t)(w_{tt}-w_{rr}) + \frac{1}{r^2}(\partial_r+\partial_t)\left[\frac{2}{p+1}\cdot\frac{|w|^{p+1}}{r^{p-1}} + |u_\theta|^2 + \frac{|u_\varphi|^2}{\sin^2 \theta}\right]\\
 & = \frac{2}{r^2}(w_r+w_t)\left(w_{tt}-w_{rr}+\frac{|w|^{p-1}w}{r^{p-1}}\right) -\frac{2(p-1)}{p+1}\cdot\frac{|w|^{p+1}}{r^{p+2}} \\
 & \qquad + \frac{1}{r^2}(\partial_r+\partial_t)\left[|u_\theta|^2 + \frac{|u_\varphi|^2}{\sin^2 \theta}\right].
\end{align*}
We recall the expression of Laplace operator in spherical coordinates
\begin{align*}
 \Delta u =  u_{rr} + \frac{2}{r} u_r + \frac{1}{r^2\sin \theta} \partial_\theta (u_\theta \sin \theta) + \frac{u_{\varphi \varphi}}{r^2 \sin^2 \theta}.
\end{align*}
This enables us to write down the equation $w$ satisfies
\begin{align*}
 w_{tt}  =  ru_{tt} & = - r|u|^{p-1} u + r\Delta u\\
 & =  -\frac{|w|^{p-1}w}{r^{p-1}} + w_{rr} + \frac{1}{r^2\sin \theta} \partial_\theta (w_\theta \sin \theta) + \frac{w_{\varphi \varphi}}{r^2 \sin^2 \theta}.
\end{align*}
Plugging this in the expression of $\hbox{div} \mathbf{V}_-$ we have
\begin{align*}
  4\; \hbox{div} \mathbf{V}_- & = \frac{2}{r^2}(w_r+w_t)\left(\frac{1}{r^2\sin \theta} \partial_\theta (w_\theta \sin \theta) + \frac{w_{\varphi \varphi}}{r^2 \sin^2 \theta}\right) -\frac{2(p-1)}{p+1}\frac{|u|^{p+1}}{|x|} \\
 & \qquad +\frac{1}{r^2} (\partial_r+\partial_t)\left(|u_\theta|^2 + \frac{|u_\varphi|^2}{\sin^2 \theta}\right)\\
 & = J_1(r,\theta,\varphi,t) + J_2(x,t) + J_3(r,\theta,\varphi,t).
\end{align*}
Now we apply Guess' formula on the surface integral in the left hand of energy flux formula and obtain
\begin{align*}
 4\int_{\partial \Omega} \mathbf{V}_- \cdot d\mathbf{S} = 4\iint_\Omega \hbox{div} \mathbf{V}_- dx dt = I_1 + I_2 + I_3
\end{align*}
The integrals $I_1$, $I_2$ and $I_3$ are defined by 
\[
 I_k = \iint_{\Omega} J_k dxdt = \iint_\Phi \int_0^\pi \int_0^{2\pi} J_k \cdot r^2 \sin \theta\, d\varphi d\theta dr dt, \qquad k=1,2,3.
\]
In particular we have
\begin{align*}
I_2  = - \frac{2(p-1)}{p+1}\iint_{\Omega} \frac{|u|^{p+1}}{|x|} dxdt.
\end{align*}
Now we deal with the most difficult term $I_1$.
\begin{align*}
I_1 & = \iint_\Phi \int_0^\pi \int_0^{2\pi} J_1(r,\theta,\varphi,t) r^2 \sin \theta\; d\varphi d\theta dr dt\\
& = \iint_\Phi \int_0^\pi \int_{0}^{2\pi} 2(w_r+w_t)\left(\frac{1}{r^2} \partial_\theta (w_\theta \sin \theta) + \frac{w_{\varphi \varphi}}{r^2 \sin \theta}\right)d\varphi d\theta dr dt\\
& = -\iint_\Phi \int_0^\pi \int_{0}^{2\pi} \frac{2}{r^2}\left[(w_{\theta r}+w_{\theta t})(w_\theta \sin \theta) + (w_{\varphi r}+w_{\varphi t})\cdot \frac{w_{\varphi}}{\sin \theta}\right]d\varphi d\theta dr dt\\
& = -\iint_\Phi \int_0^\pi \int_{0}^{2\pi} \frac{1}{r^2} \left[(\partial_r + \partial_t)\left(|w_\theta|^2 + \frac{|w_{\varphi}|^2}{\sin^2 \theta}\right)\right] \sin \theta\; d\varphi d\theta dr dt\\
& = -\iint_\Phi \int_0^\pi \int_{0}^{2\pi} \frac{1}{r^2} \left\{(\partial_r + \partial_t)\left[r^2\left(|u_\theta|^2 + \frac{|u_{\varphi}|^2}{\sin^2 \theta}\right)\right]\right\} \sin \theta\; d\varphi d\theta dr dt\\
& = -\iint_\Phi \int_0^\pi \int_{0}^{2\pi} \left[\frac{2}{r} \left(|u_\theta|^2 + \frac{|u_{\varphi}|^2}{\sin^2 \theta}\right) +  r^2 J_3(r,\theta,\varphi,t) \right] \sin \theta\; d\varphi d\theta dr dt\\
& = - \iint_\Omega \frac{2|\slashed{\nabla} u|^2}{|x|} dx dt - I_3 
\end{align*}
In the final step we use the expression of $|\slashed{\nabla} u|^2$ in spherical coordinates \eqref{expression of slashed}. Collecting all the terms $I_1$, $I_2$ and $I_3$ we have
\[
 4\int_{\partial \Omega} \mathbf{V}_- \cdot d\mathbf{S} = - \frac{2(p-1)}{p+1}\iint_{\Omega} \frac{|u|^{p+1}}{|x|} dxdt- \iint_\Omega \frac{2|\slashed{\nabla} u|^2}{|x|} dx dt.
\]
This proves the energy flux formula when $\Phi$ is away from the $t$-axis.

\paragraph{The case with boundary on $t$-axis} Now let us consider the case when part of the boundary is on the $t$-axis. In this case a limit process $r \rightarrow 0^+$ is required. More precisely we first apply energy flux formula away from $t$-axis on the region $\Omega_r = \Omega \cap \{(x,t)\in \Rm^3 \times \Rm: |x|\geq r\}$ and then make $r\rightarrow 0^+$. In order to complete the proof we only need to show that there exists a positive, nonnegative, continuous measure so that the following identity holds for all $t_1<t_2$ and $c_1,c_2\in \{-1,0,1\}$.
\begin{equation*} 
  \lim_{r\rightarrow 0^+} \int_{S(t_1+c_1r,t_2+c_2r,r)} \mathbf{V}_- \cdot d\mathbf{S}  = +\pi \mu([t_1,t_2]).
\end{equation*}
Here $S(t_1,t_2,r)=\{(x,t): |x|=r, t_1<t<t_2\}$ is a circular cylinder oriented inward. Please see figure \ref{figure limit} for an illustration of this limit process in three different cases with $c_2=-1,0,1$, respectively. The geometric meaning of $c_1$ is similar. We may plug in the definition of $\mathbf{V}_-$ and write the surface integral above in details.
\begin{equation} \label{limit of vertical integral}
 \lim_{r\rightarrow 0^+} \int_{S(t_1+c_1r,t_2+c_2r,r)} \left(\frac{1}{4} \left|\mathbf{L}_+ u\right|^2 - \frac{1}{4}\left|\slashed{\nabla} u\right|^2 - \frac{|u|^{p+1}}{2(p+1)}\right) dS = \pi \mu([t_1,t_2]).
\end{equation}

\begin{figure}[h]
 \centering
 \includegraphics[scale=1.2]{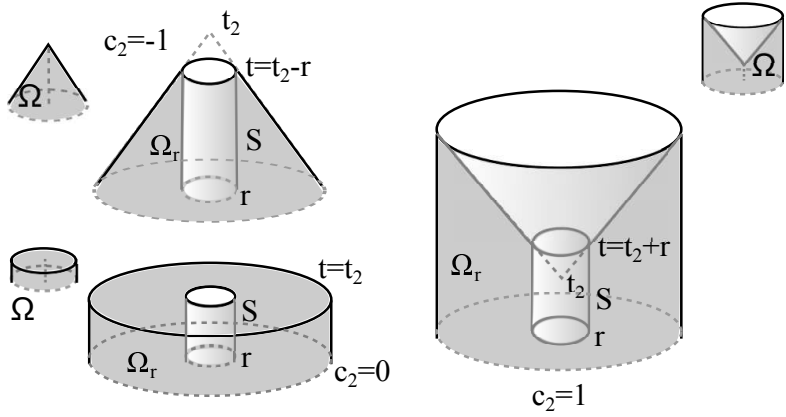}
 \caption{Illustration of the limit process $r \rightarrow 0^+$} \label{figure limit}
\end{figure}

\noindent We first define 
\[
 P(t) = \lim_{r\rightarrow 0^+} \int_{S(0,t,r)} \left(\frac{1}{4} \left|\mathbf{L}_+ u\right|^2 - \frac{1}{4}\left|\slashed{\nabla} u\right|^2 - \frac{|u|^{p+1}}{2(p+1)}\right)  dS.
\]
If $t<0$, we integrate on the surface $S(t,0,r)$ but multiply the integral by $-1$. The proof consists of four steps. 
\begin{itemize}
 \item[(1)] The limit in the left hand side of \eqref{limit of vertical integral} always converges. Thus the function $P(t)$ is well defined. It is clear we then have 
 \[
  \lim_{r\rightarrow 0^+} \int_{S(t_1,t_2,r)} \left( \frac{1}{4} \left|\mathbf{L}_+ u\right|^2 - \frac{1}{4}\left|\slashed{\nabla} u\right|^2 - \frac{|u|^{p+1}}{2(p+1)}\right) dS = P(t_2) - P(t_1). 
\] 
 \item[(2)] The function $P(t)$ is continuous and nondecreasing. In addition\footnote{Since $P(t)$ is nondecreasing, $P(\pm \infty)$ represents its limits at infinity} $P(+\infty)-P(-\infty) \lesssim_p E$.
 \item[(3)] The function $P(t)$ defines a nonnegative, continuous, finite measure by $\pi \mu([t_1,t_2]) = P(t_2)-P(t_1)$. Thus identity \eqref{limit of vertical integral} holds for $c_1=c_2=0$.
 \item[(4)] We prove that identity \eqref{limit of vertical integral} also holds for other choices of $c_1,c_2$. 
\end{itemize}
\paragraph{Step 1} We consider the region $\Omega_r = \{(x,t): r\leq |x|\leq r_0, t_1+c_1|x|\leq t\leq t_2+c_2|x|\}$ with $0<r<r_0$, which is away from $t$-axis. Thus we may apply energy flux formula on this region. 
\begin{align*}
 & \int_{S(t_1+c_1r,t_2+c_2r,r)} \left( \frac{1}{4} \left|\mathbf{L}_+ u\right|^2 - \frac{1}{4}\left|\slashed{\nabla} u\right|^2 - \frac{|u|^{p+1}}{2(p+1)}\right) dS + \int_{\Sigma_1(r)} \mathbf{V}_-\cdot d\mathbf{S}\\
 & \quad + \int_{\Sigma_2(r)} \mathbf{V}_- \cdot d\mathbf{S} + \int_{\Sigma_3} \mathbf{V}_-\cdot d\mathbf{S} = - \iint_{\Omega_r} \left(\frac{p-1}{2(p+1)}\frac{|u|^{p+1}}{|x|} + \frac{1}{2} \frac{|\slashed{\nabla} u|^2}{|x|} \right) dx dt.
\end{align*}
Here $\Sigma_j(r) = \{(x,t): t=t_j+c_j|x|, r\leq |x|\leq r_0\}, \, j=1,2$ and $\Sigma_3 = \{(x,t): |x|=r_0, t_1+c_1r_0\leq t \leq t_2+c_2r_0\}$.  The first term is the integral we are interested in. All the other terms in the identity above converge as $r\rightarrow 0^+$. Because
\begin{itemize}
 \item The integral on $\Sigma_j, \;j=1,2$ converges as $r\rightarrow 0^+$ by either Remark \ref{bound of Q}, if $c_j=\pm 1$, or finiteness of energy, if $c_j=0$. 
 \item The integral on $\Sigma_3$ is independent to $r$.
 \item The double integral in the right hand side converges by Morawetz estimates.
\end{itemize}
As a result the first term above converges as $r \rightarrow 0^+$.
\paragraph{Step 2} We first show that $P(t)$ is continuous. Without loss of generality, let us assume $t>0$. If we choose the region $\Omega(r,t) = \{(x,t'): r\leq |x|\leq 1, 0\leq t'\leq t\}$, apply energy flux formula and let $r\rightarrow 0^+$, we obtain
\begin{align*}
 P(t) + E_-(t;0,1) - E_-(0;0,1) + & \int_{S(0,t,1)} \left(-\frac{1}{4} \left|\mathbf{L}_+ u\right|^2 + \frac{1}{4}\left|\slashed{\nabla} u\right|^2 + \frac{|u|^{p+1}}{2(p+1)}\right) dS\\
 & =   -\iint_{\Omega(t)} \left(\frac{p-1}{2(p+1)}\frac{|u|^{p+1}}{|x|} + \frac{1}{2} \frac{|\slashed{\nabla} u|^2}{|x|} \right) dx dt.
\end{align*}
Here $\Omega(t) = \{(x,t'): |x|\leq 1, 0\leq t'\leq t\}$. Because all other terms are continuous in $t$, we obtain the continuity of $P(t)$. In order to show the monotonicity, we only need to show
\[
 P(t_2) - P(t_1) = \lim_{r\rightarrow 0^+} \int_{S(t_1,t_2,r)} \left(\frac{1}{4} \left|\mathbf{L}_+ u\right|^2 - \frac{1}{4}\left|\slashed{\nabla} u\right|^2 - \frac{|u|^{p+1}}{2(p+1)}\right)  dS \geq 0
\]
for all $t_2 > t_1$. We first observe 
\[
 \int_{0}^\infty\left[\frac{1}{r} \int_{S(-\infty,\infty,r)} \left(|u|^{p+1} + \left|\slashed{\nabla} u\right|^2 \right)dS\right] dr = \iint_{\Rm^3\times \Rm} 
 \frac{|u|^{p+1}+\left|\slashed{\nabla} u\right|^2}{|x|} dx dt \lesssim_p E.
\]
Thus there exists a sequence $r_n \rightarrow 0^+$ so that 
\[
 \lim_{n\rightarrow \infty} \int_{S(-\infty,\infty,r_n)} \left(|u|^{p+1} + \left|\slashed{\nabla} u\right|^2 \right)dS = 0.
\]
As a result, we have the monotonicity 
\[
 P(t_2) - P(t_1) = \lim_{n\rightarrow \infty} \frac{1}{4}\int_{S(t_1,t_2,r_n)} \left|\mathbf{L}_+ u\right|^2 dS \geq 0.
\]
Similarly we observe 
\begin{align*}
 & \int_{0}^R \int_{S(-\infty,\infty,r)} \left(|u|^{p+1} + \left|\slashed{\nabla} u\right|^2 + \left|\mathbf{L}_+ u\right|^2 \right)dS dr\\
 \lesssim_1 & \int_{-\infty}^\infty \int_{|x|<R} \left(|u|^{p+1} + \left|\slashed{\nabla} u\right|^2 + \left|\mathbf{L} u \right|^2 + |u_t|^2\right) dx dt\\
 \lesssim_1 & \int_{-\infty}^\infty \int_{|x|<R} \left(|u|^{p+1}  + |\nabla u|^2 + |u_t|^2\right) dx dt + \frac{1}{R}\int_{-\infty}^\infty \int_{|x|=R} |u|^2 d\sigma_R(x) dt\\
 \lesssim_p & RE.
\end{align*}
Here we use Remark \ref{relationship of u w energy} and Corollary \ref{cor u morawetz}. The inequality above immediately gives
\[
 \liminf_{r \rightarrow 0^+}\int_{S(-\infty,\infty,r)} \left(|u|^{p+1} + \left|\slashed{\nabla} u\right|^2 + \left|\mathbf{L}_+ u\right|^2 \right)dS \lesssim_p E.
\]
This implies $P(+\infty)-P(-\infty) \lesssim_p E$. 
\paragraph{Step 3} Now we collect all information about $P(t)$ obtained in step 2. According to measure theory, there exists a nonnegative, continuous, finite Borel measure $\mu$, so that $\mu([t_1,t_2]) = P(t_2) - P(t_1)$. Readers may refer to Tao \cite{taomeasure} for related measure theory. 
\paragraph{Step 4} Let us recall the sequence $r_n$ introduced in Step 2. The integral of $|\slashed{\nabla} u|^2 + |u|^{p+1}$ over cylinder $\{(x,t): |x|=r_n\}$ can be ignored if we consider the limit as $n\rightarrow +\infty$. Thus for any constant $\eps>0$ we have
\begin{align*}
 & \lim_{r\rightarrow 0^+} \int_{S(t_1+c_1r,t_2+c_2r,r)} \left(\frac{1}{4} \left|\mathbf{L}_+ u\right|^2 - \frac{1}{4}\left|\slashed{\nabla} u\right|^2 - \frac{|u|^{p+1}}{2(p+1)}\right)  dS\\
 = & \lim_{n\rightarrow \infty} \frac{1}{4}\int_{S(t_1+c_1r_n,t_2+c_2r_n,r_n)} \left|\mathbf{L}_+ u\right|^2 dS\\
 \leq & \lim_{n\rightarrow \infty} \frac{1}{4} \int_{S(t_1-\eps,t_2+\eps,r_n)} \left|\mathbf{L}_+ u\right|^2  dS\\
 = & \lim_{n \rightarrow \infty} \int_{S(t_1-\eps+0\cdot r_n,t_2+\eps+0\cdot r_n,r_n)} \left(\frac{1}{4} \left|\mathbf{L}_+ u\right|^2 - \frac{1}{4}\left|\slashed{\nabla} u\right|^2 - \frac{|u|^{p+1}}{2(p+1)}\right)  dS\\
 = & \pi \mu([t_1-\eps,t_2+\eps]). 
\end{align*}
In the same way we can find a lower bound $\pi\mu([t_1+\eps, t_2-\eps])$ for the limit above. Finally we make $\eps\rightarrow 0^+$ and obtain the desired  identity \eqref{limit of vertical integral} by the continuity of the measure $\mu$.

\begin{remark}
If we follow a similar limit process for the energy flux of outward energy, we will obtain the same measure $\mu$ as in the case of inward energy. There are two different ways to show this. 
 \begin{itemize}
 \item We may apply smooth approximation techniques and use the following fact: If the solution $u(x,t)$ is sufficiently smooth near the origin, then we obtain $d \mu(t) = |u(0,t)|^2 dt$ in both limit processes by the following calculation
 \begin{equation*}
  \lim_{r\rightarrow 0^+} \int_{|x|=r} \left|\nabla u\cdot \frac{x}{|x|} + \frac{u}{|x|} \pm u_t\right|^2 d\sigma_r (x)  = 4\pi |u(0,t)|^2
 \end{equation*}
 \item Let us temporarily use notations $\mu_-$ (inward case) and $\mu_+$ (outward case) for measures we obtained in the limit process above. Following the same argument as in Section \ref{sec:distribution}, we have
 \begin{align*}
  E_-(t_2) - E_-(t_1) & = -\pi \mu_-([t_1,t_2])  - \mathcal{M}(\Rm^3 \times [t_1,t_2]);\\
  E_+(t_2)-E_+(t_1) & = +\pi \mu_+([t_1,t_2]) + \mathcal{M}(\Rm^3 \times [t_1,t_2]).
 \end{align*}
 A combination of these identities with the fact $E_+(t) + E_-(t) = E$ shows $\mu_-([t_1,t_2]) = \mu_+([t_1,t_2])$ for all $t_1<t_2$.
 \end{itemize}
\end{remark}

\subsection{Energy Flux Formula for Cones}
We can apply our general energy flux formula on a cone region. This will be frequently used in Section \ref{sec:distribution}. 
\begin{proposition}[Cone Law] \label{cone law}
Given any $s_0>t_0$, we can define $\Omega = \{(x,t):t\geq t_0, |x|+t\leq s_0\}$ to be a cone region and obtain the following identity from energy flux formula
\[
 E_-(t_0;0,s_0-t_0) = \pi \mu([t_0,s_0]) + Q_-^-(s_0;t_0,s_0) + \mathcal{M}(\Omega).
\]
 We can substitute $s_0$ by $t_0+r_0$ with $r_0>0$ and rewrite this in the form
\[
 E_-(t_0;0,r_0) = \pi \mu([t_0,t_0+r_0]) + Q_-^-(t_0+r_0;t_0,t_0+r_0) + \mathcal{M}(\Omega).
\]
\end{proposition}

\section{Energy Distribution of Solutions} \label{sec:distribution}

In this section we apply the inward/outward energy method to prove Theorem \ref{main 1} and then proves the weighted Morawetz estimates. Throughout this section we assume that $3\leq p\leq 5$ and $u$ is a solution to (CP1) with a finite energy $E$.

\subsection{Asymptotic Behaviour of Inward and Outward Energies}

\paragraph{An upper bound on $\mu(\Rm)$} All inward and outward energies are clearly bounded above by the full energy $E$, since $E= E_-(t)+E_+(t)$. We also know that all the energy fluxes $Q$'s are smaller or equal to $E$, according to Remark \ref{bound of Q}. Let us give a similar upper bound of $\mu(\Rm)$.
\begin{lemma}[Measure bound] \label{L2 bound}
 Let $u$ be a solution to (CP1) with an energy $E$. Then the measure $\mu$ in the energy flux formula satisfies $\pi \mu(\Rm) \leq E$.
\end{lemma}
\begin{proof}
We apply cone law on the region $\Omega(s,t_0) = \{(x,t): |x|+t\leq s, t\geq t_0\}$ for any $t_0<s$ and obtain
\[
 E_-(t_0;0,s-t_0) = \pi \mu([t_0,s]) + Q_-^-(s;t_0,s) + \mathcal{M}(\Omega(s,t_0)),
\]
as shown in figure \ref{figure finitemu}. Making $t_0 \rightarrow -\infty$ we have 
\[
 \pi \mu((-\infty,s]) + Q_-^-(s) + \mathcal{M}(\{(x,t): |x|+t \leq s\}) \leq E.
\]
Finally we let $s \rightarrow +\infty$ in the inequality above to obtain $\pi \mu(\Rm) \leq E$. 
\end{proof}

 \begin{figure}[h]
 \centering
 \includegraphics[scale=0.9]{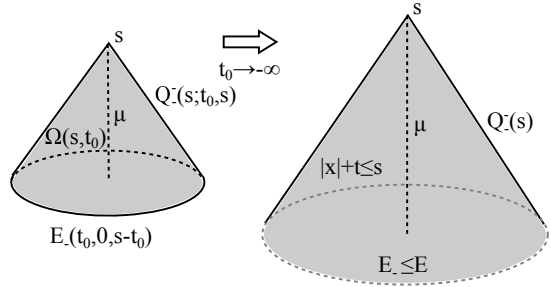}
 \caption{Illustration for proof of Lemma \ref{L2 bound}} \label{figure finitemu}
\end{figure}

\paragraph{Monotonicity and asymptotic behaviour} Next we consider the monotonicity and asymptotic behaviour of inward and outward energies as $t \rightarrow \pm \infty$. Let us first give a technical lemma.

\begin{lemma} \label{lift of r}
 Given any $t_0\in \Rm$ and $0<r_1<r_2$, we have
 \[
   \int_{r_1}^{r_2} Q_-^-(t_0+r;t_0,t_0+r) dr \leq r_2 \mathcal{M}(\Omega(t_0,r_1,r_2)).
 \]
 Here $\Omega(t_0,r_1,r_2) = \{(x,t): t\geq t_0,t_0+r_1\leq t+|x|\leq t_0+r_2\}$ is a cone shell region. In particular we have the following lower limit for any fixed $t_0\in \Rm$,
 \[
  \liminf_{r \rightarrow +\infty} Q_-^-(t_0+r;t_0,t_0+r) = 0.
 \]
\end{lemma}

 \begin{figure}[h]
 \centering
 \includegraphics[scale=0.75]{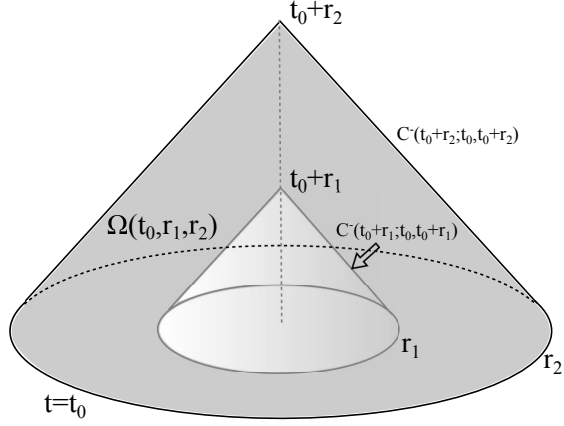}
 \caption{Illustration of integral region} \label{upperQ}
\end{figure}

\begin{proof}
 We may recall the definition of $Q_-^-$ and conduct a straightforward calculation
\begin{align*}
 \int_{r_1}^{r_2} Q_{-}^{-} (t_0+r;t_0,t_0+r) dr & = \frac{1}{\sqrt{2}} \int_{r_1}^{r_2} \int_{C^-(t_0+r; t_0, t_0+r)} \left(\frac{1}{p+1}|u|^{p+1} +  \frac{1}{2}\left|\slashed{\nabla} u\right|^2 \right) dS dr \\
 & =  \iint_{\Omega(t_0,r_1,r_2)} \left(\frac{1}{p+1}|u|^{p+1} +  \frac{1}{2}\left|\slashed{\nabla} u\right|^2 \right) dx dt\\
 & \leq r_2 \iint_{\Omega(t_0,r_1,r_2)} \left(\frac{p-1}{2(p+1)}\cdot \frac{|u|^{p+1}}{|x|} + \frac{1}{2}\cdot \frac{\left|\slashed{\nabla} u\right|^2}{|x|}\right) dx dt\\
 & = r_2 \mathcal{M}(\Omega(t_0,r_1,r_2)).
 \end{align*}
Here we use the fact $\Omega(t_0,r_1,r_2)  \subset B(0,r_2)\times \Rm$, as shown in figure \ref{upperQ}. In order to prove the lower limit, we choose $r_2 = 2r_1$ in the integral estimate above and apply the mean value theorem. We obtain
\[
 \inf_{r\in [r_1,2r_1]} Q_-^-(t_0+r;t_0,t_0+r) \leq \frac{1}{r_1} \int_{r_1}^{2r_1}  Q_-^-(t_0+r;t_0,t_0+r) dr  \leq 2 \mathcal{M}(\Omega(t_0,r_1,2r_1)).
\]
Finally we observe the fact $\mathcal{M}(\Omega(t_0,r_1,2r_1)) \rightarrow 0$ as $r_1\rightarrow +\infty$ to finish the proof.
\end{proof}

\noindent Now we are ready to prove the monotonicity and limits of inward/outward energies. This gives Part (a) of Theorem \ref{main 1}.

\begin{proposition}[Monotonicity and limits] \label{monotonicity of energies}
 The inward energy $E_-(t)$ is a decreasing function of $t$, while the outward energy $E_+(t)$ is an increasing function of $t$. In addition 
\begin{align*}
 &\lim_{t\rightarrow +\infty} E_-(t) = 0; & &\lim_{t\rightarrow -\infty} E_+(t) = 0.&
\end{align*}
\end{proposition}
\begin{proof}
Let us prove the monotonicity and limit of inward energy. The outward energy can be dealt with in the same way. First of all, we apply inward energy flux formula on the truncated cone region $\Omega(t_1,t_2,s) = \{(x,t): t_1\leq t\leq t_2, |x|+t\leq s\}$ for $s\geq t_2>t_1$, as shown in the left upper corner of figure \ref{monotonicity}. 
\begin{align}
  E_-(t_2; 0, s-t_2) - E_-(t_1;0,s-t_1) = -\pi \mu([t_1,t_2]) - Q_-^-(s;t_1,t_2) - \mathcal{M}(\Omega(t_1,t_2,s)). \label{trapezoid app 1}
\end{align}
By Lemma \ref{lift of r}, we have 
\[
 \liminf_{s \rightarrow +\infty} Q_-^-(s;t_1,t_2) \leq \liminf_{s \rightarrow +\infty} Q_-^-(s;t_1,s) =0.
\]
Therefore we can make $s\rightarrow \infty$ in the identity above and obtain the monotonicity
\[
 E_-(t_2) - E_-(t_1) = -\pi \mu([t_1,t_2])  - \mathcal{M}(\Rm^3 \times [t_1,t_2]) < 0.
\]
This identity can be understood as the energy flux formula on the unbounded region $\Rm^3 \times [t_1,t_2]$, as shown in the right upper corner of figure \ref{monotonicity}. Please note that the rectangular boxes in this figure (and subsequent figures) represent unbounded regions. The hollow arrows besides the rectangular boxes indicate the directions in which the region may extend infinitely. Next we apply cone law on $\Omega(t_0,r_0) =\{(x,t): |x|+t\leq t_0+r_0, t\geq t_0\}$ with $t_0\in \Rm$ and $r_0>0$, as shown in the lower part of figure \ref{monotonicity}
\begin{align*}
 E_-(t_0;0,r_0) & = \pi \mu([t_0,t_0+r_0]) + Q_-^-(t_0+r_0; t_0, t_0+r_0) + \mathcal{M}(\Omega(t_0,r_0)).
\end{align*}
According to Lemma \ref{lift of r}, we can take a limit $r_0 \rightarrow +\infty$ and give an expression of $E_-(t_0)$ in terms of $\mu$ and Morawetz integral.
\begin{equation}
 E_-(t_0)  = \pi \mu([t_0,\infty)) +  \mathcal{M}(\Rm^3 \times [t_0,\infty)). \label{expression of E neg temp}
\end{equation}
Finally we can make $t_0\rightarrow +\infty$ and finish the proof.
\end{proof}

 \begin{figure}[h]
 \centering
 \includegraphics[scale=0.95]{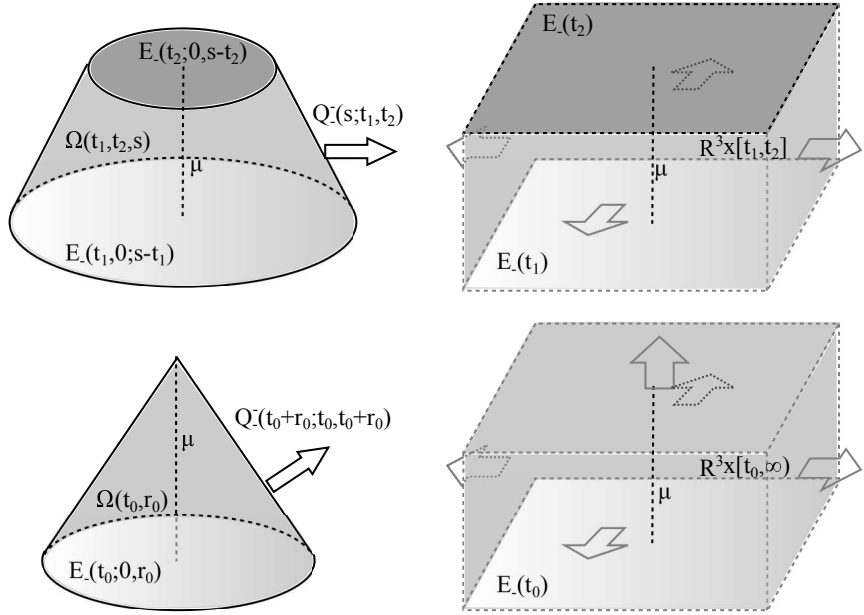}
 \caption{Illustration for proof of Proposition \ref{monotonicity of energies}} \label{monotonicity}
\end{figure}

\begin{corollary} \label{asymptotic behaviour 2}
 We have the limits 
 \begin{align*}
  &\lim_{t\rightarrow -\infty} E_-(t) = E; &  &\lim_{t\rightarrow +\infty} E_+(t) = E;&\\
  &\lim_{t\rightarrow \pm \infty} \int_{\Rm^3} \left(|u(x,t)|^{p+1} + \left|\slashed{\nabla} u(x,t)\right|^2\right) dx = 0.&  &&
 \end{align*}
\end{corollary}

\begin{remark} \label{expression of E neg}
Making $t \rightarrow -\infty$ in identity \eqref{expression of E neg temp} we have
\begin{equation*}
  \pi \mu(\Rm) +  \mathcal{M}(\Rm^3 \times \Rm) = E.
\end{equation*}
This means that all the energy eventually changes from inward energy to outward energy in either of the following ways
\begin{itemize}
 \item Inward waves travel through the origin and become outward waves. The amount of energy carried by these inward waves is $\pi \mu(\Rm)$.
 \item Inward energy is transformed to outward energy by the effect of wave propagation at all times and places. The total amount of energy transformed in this way is exactly the integral $\mathcal{M}(\Rm^3 \times \Rm)$.
\end{itemize}
\end{remark}

\noindent We also have asymptotic behaviour of energy flux
\begin{proposition} \label{a limit of Q}
 We have the limits
 \begin{align*}
  &\lim_{s \rightarrow +\infty} Q_-^-(s) = 0;& &\lim_{\tau \rightarrow -\infty} Q_+^+(\tau) = 0.&
 \end{align*}
\end{proposition}
\begin{proof}
 Again we only give a proof for inward energy flux. An application of inward energy flux on the region $\Omega(t_0,s,s') = \{(x,t): t_0\leq t\leq s, s\leq |x|+t\leq s'\}$ with $t_0<s<s'$ gives (Please see left upper part of figure \ref{limitQ})
\begin{align}
  E_-(s; 0, s'-s) - E_-(t_0;s-t_0,s'-t_0) = Q_-^-(s;t_0,s)- Q_-^-(s';t_0,s) - \mathcal{M}(\Omega(t_0,s,s')). 
\end{align} 
Making $s'\rightarrow \infty$ we can discard the term $Q_-^-(s';t_0,s)$ as in the proof of Proposition \ref{monotonicity of energies} and rewrite the identity above into
\begin{align*}
 E_-(s) + \mathcal{M}(\{(x,t): |x|+t\geq s, t_0\leq t\leq s\}) = E_-(t_0; s-t_0,\infty) + Q_-^-(s;t_0,s),
\end{align*}
as shown in left lower part of figure \ref{limitQ}. Next we can take a limit as $t_0\rightarrow - \infty$.
\[
 E_-(s) +  \mathcal{M}(\{(x,t): |x|+t\geq s, t\leq s\}) = \lim_{t\rightarrow -\infty} E_-(t;s-t,\infty) + Q_-^-(s).
\]
Please see right half of figure \ref{limitQ}. Finally we observe that both terms in left hand  converges to zero as $s \rightarrow +\infty$ and finish the proof. 
\end{proof}

 \begin{figure}[h]
 \centering
 \includegraphics[scale=1.2]{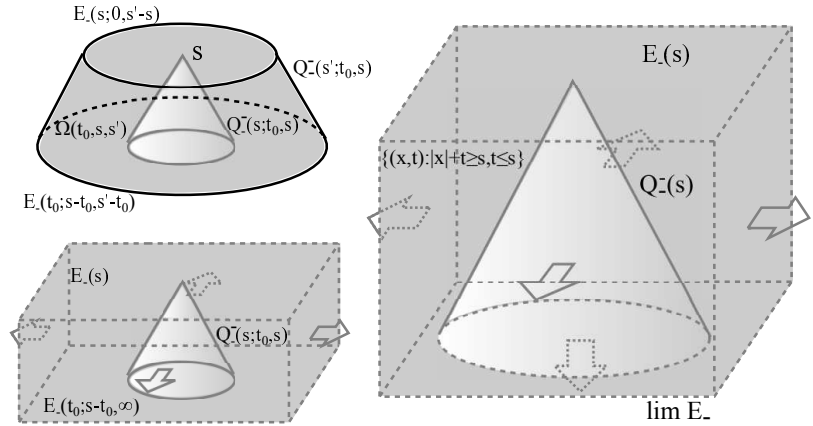}
 \caption{Illustration for proof of Proposition \ref{a limit of Q}} \label{limitQ}
\end{figure}

\paragraph{Asymptotic travel speed of energy} Finally we prove that the energy eventually travels to the infinity at a speed at least close to the light speed as $t \rightarrow \pm \infty$.

\begin{proposition} \label{inner decay}
 For any constant $c \in (0,1)$, we have
 \[
  \lim_{t \rightarrow \pm \infty} E_{\pm} (t; 0, c|t|) = 0.
 \]
\end{proposition}

 \begin{figure}[h]
 \centering
 \includegraphics[scale=0.8]{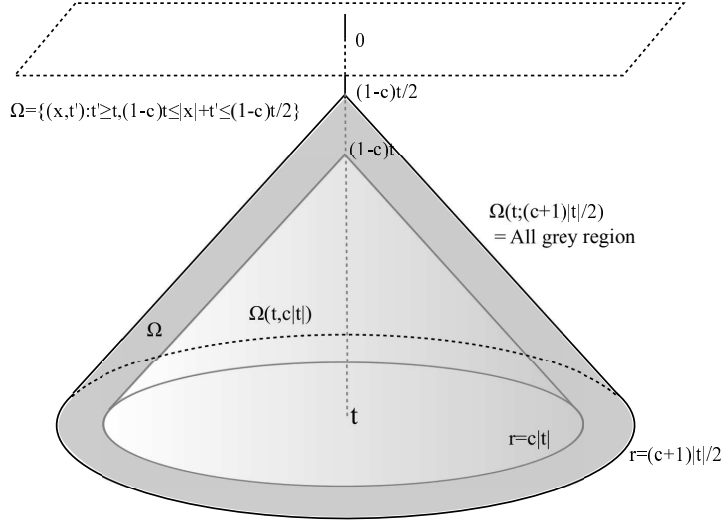}
 \caption{Illustration for proof of Proposition \ref{inner decay}} \label{figure tct}
\end{figure}

\begin{proof}
 Without loss of generality we prove the limit of inward energy. Figure \ref{figure tct} gives an illustration of the proof and may be helpful for readers. We start by
 \begin{equation}
  E_-(t;0,c|t|) \leq E_-(t;0,r),\; \forall r>c|t|;\; \Rightarrow \; E_-(t;0,c|t|) \leq \frac{2}{(1-c)|t|} \int_{c|t|}^{\frac{c+1}{2}|t|} E_-(t;0,r) dr. \label{upper bound ct}
 \end{equation}
 We then apply the cone law on $\Omega(t,r) = \{(x,t'): t'\geq t, t'+|x|\leq t+r\}$ with $r\in [c|t|,(c+1)|t|/2]$.
 \begin{align*}
  E_-(t;0,r) & = \pi \mu([t,t+r]) + Q_-^-(t+r;t,t+r) + \mathcal{M}(\Omega(t,r)) \\
   & \leq \sup_{r \in \left(c|t|, \frac{c+1}{2}|t|\right)} \Big\{\pi \mu([t,t+r]) + \mathcal{M}(\Omega(t,r))\Big\} + Q_-^-(t+r;t,t+r)\\
   & = \pi \mu([t,(1-c)t/2]) + \mathcal{M}(\Omega(t,(c+1)|t|/2)) + Q_-^-(t+r;t,t+r)\\
   & = P(t) + Q_-^-(t+r;t,t+r).
 \end{align*}
 Here the function $P(t) =  \pi \mu([t,(1-c)t/2]) + \mathcal{M}(\Omega(t,(c+1)|t|/2)) \rightarrow 0$ as $t\rightarrow -\infty$. We may plug this upper bound of $E_-(t;0,r)$ in the inequality \eqref{upper bound ct} and obtain 
 \begin{align}
  E_-(t;0,c|t|) & \leq P(t) + \frac{2}{(1-c)|t|} \int_{c|t|}^{\frac{c+1}{2}|t|} Q_-^-(t+r;t,t+r) dr \nonumber \\
   & \leq P(t) + \frac{1+c}{1-c} \mathcal{M}\left(\left\{(x,t'): t'\geq t, (1-c)t\leq t'+|x|\leq \frac{1-c}{2}t\right\}\right). \label{final upper bound1}
 \end{align}
 Here we apply Lemma \ref{lift of r}. Because both terms in the right hand side of inequality \eqref{final upper bound1} converge to zero as $t \rightarrow -\infty$, we finally finish the proof.
\end{proof}
\noindent Before we conclude this subsection, we verify Part (b) of Theorem \ref{main 1}. 
\begin{corollary}
 Given any constant $c \in (0,1)$, we have
 \[
  \lim_{t \rightarrow \pm \infty} \int_{|x|<c|t|} \left(\frac{1}{2}|\nabla u(x,t)|^2 + \frac{1}{2}|u_t(x,t)|^2 + \frac{1}{p+1}|u(x,t)|^{p+1}\right) dx = 0.
 \]
\end{corollary}
\begin{proof}
 By the definition of $E_\pm$ and Remark \ref{relationship of u w energy} we have
 \begin{align*}
 \int_{|x|<c|t|} \left(\frac{1}{2}|\nabla u|^2 + \frac{1}{2}|u_t|^2 + \frac{1}{p+1}|u|^{p+1}\right) dx \leq E_-(t; 0, c|t|) + E_+(t; 0, c|t|).
 \end{align*}
 The right hand converges to zero as $t \rightarrow \pm \infty$, thanks to Proposition \ref{monotonicity of energies} and Proposition \ref{inner decay}.
\end{proof}

\subsection{Weighted Morawetz Estimates}
\begin{proposition}[General Weighted Morawetz] \label{Weighted Morawetz General}
 Let $3\leq p<5$ and $0<\gamma<1$. Assume that the function $a(r) \in C([0,\infty))$ satisfies
\begin{itemize}
 \item $a(r)$ is absolutely continuous in $[0,R]$ for any $R \in \Rm^+$;
 \item Its derivative satisfies $0 \leq a'(r) \leq \gamma a(r)/r$ almost everywhere $r>0$.
\end{itemize}
 If $u$ is a solution to (CP1) with a finite energy so that 
 \begin{align*}
  K_1 =  \int_{\Rm^3} a(|x|) \left[\frac{1}{4}\left|(\mathbf{L}_+ (u_0,u_1))(x)\right|^2 + \frac{1}{4}|\slashed{\nabla} u_0(x)|^2 + \frac{1}{2(p+1)}|u_0(x)|^{p+1} \right] dx < \infty,
 \end{align*}
 then we have 
 \begin{align*}
   \pi \int_0^\infty a(t) d\mu(t) + \iint_{\Rm^3 \times \Rm^+} a(|x|+t)\left(\frac{p-1-2\gamma}{2(p+1)}\cdot\frac{|u|^{p+1}}{|x|} + \frac{1-\gamma}{2} \cdot\frac{|\slashed{\nabla} u|^2}{|x|} \right) dxdt \leq K_1,
 \end{align*}
 where $\mu$ is the measure in the energy flux formula.
\end{proposition} 

 \begin{figure}[h]
 \centering
 \includegraphics[scale=1.1]{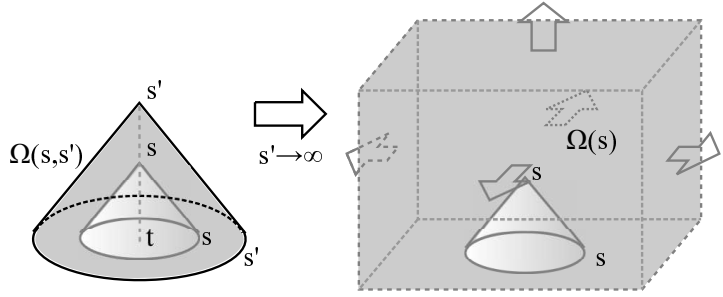}
 \caption{Illustration of regions in the proof of Proposition \ref{Weighted Morawetz General}} \label{figure weightM}
\end{figure}

\begin{proof}
According to Proposition \ref{global existence finite energy}, the solution $u$ is defined for all time $t \in \Rm$. We apply the energy flux formula of inward energy on the region $\Omega(s,s')=\{(x,t): s\leq |x|+t\leq s', t\geq 0\}$, as shown in figure \ref{figure weightM}, and obtain ($s<s'$)
 \begin{align*}
    E_-(0,s,s')  - Q_-^-(s',0,s')-\pi \mu([s,s']) +Q_-^-(s; 0,s)  - \mathcal{M}(\Omega(s,s')) = 0.
 \end{align*}
 Next we recall $\displaystyle \lim_{s'\rightarrow +\infty} Q_-^-(s') = 0$ by Proposition \ref{a limit of Q}, make $s'\rightarrow +\infty$ and obtain the energy flux formula for unbounded region $\Omega(s)=\{(x,t):t\geq 0, |x|+t\geq s\}$.
 \[
  E_-(0,s,\infty)  -\pi \mu([s,\infty)) +Q_-^-(s; 0,s)  - \mathcal{M}(\Omega(s))= 0.
 \]
 We move the terms with a negative sign to the other side of the identity, then write down the details about energy flux $Q$, inward energy $E_-$ and Morawetz integral $\mathcal{M}$.
 \begin{align}
  &\pi \mu([s,\infty)) + \iint_{\Omega(s)} \left(\frac{p-1}{2(p+1)}\cdot \frac{|u|^{p+1}}{|x|} + \frac{1}{2} \cdot\frac{|\slashed{\nabla} u|^2}{|x|} \right) dx dt \nonumber\\
  &\quad =  \int_{|x|>s} e_-(x,0) dx + \frac{1}{\sqrt{2}} \int_{C^-(s;0,s)} \left(\frac{1}{p+1}|u|^{p+1} +  \frac{1}{2}\left|\slashed{\nabla} u\right|^2 \right) dS. \label{energy flux for omega s}
 \end{align}
Here we use the following notation for convenience. 
\[
 e_-(x,0) = \frac{1}{4}\left|(\mathbf{L}_+ (u_0,u_1))(x)\right|^2 + \frac{1}{4}|\slashed{\nabla} u_0(x)|^2 + \frac{1}{2(p+1)}|u_0(x)|^{p+1}.
\]
We then multiply both sides by $a'(s)$ and integrate for $s$ from $0$ to $R\gg 1$
\begin{align}
 & \pi \! \int_0^\infty \![a_R(t)\!-\!a_R(0)] d\mu(t) + \iint_{\Rm^3\times \Rm^+}\! [a_R(|x|+t)\!-\!a_R(0)]\left(\frac{p-1}{2(p+1)}\cdot \frac{|u|^{p+1}}{|x|} + \frac{1}{2} \cdot\frac{|\slashed{\nabla} u|^2}{|x|} \right) dxdt \nonumber\\
 &\quad =  \int_{\Rm^3} [a_R(|x|)-a_R(0)] e_-(x,0) dx + \iint_{\Omega(0,R)} a'(|x|+t)\left[\frac{1}{p+1}|u|^{p+1} +  \frac{1}{2}\left|\slashed{\nabla} u\right|^2 \right]  dxdt. \label{integral identity in weighted 1}
\end{align}
Here $a_R$ is a truncated version of $a$ defined by 
\[
 a_R(x) = \left\{\begin{array}{ll} a(x), & \hbox{if}\; 0\leq x<R;\\ a(R), &\hbox{if}\; x\geq R. \end{array}\right.
\]
Next we recall the expression of $E_-(0)$ in term of $\mu$ and the Morawetz integral, as given in \eqref{expression of E neg temp} and rewrite it in the form of 
\[
 \pi \mu([0,\infty)) + \iint_{\Rm^3 \times \Rm^+} \left(\frac{p-1}{2(p+1)}\cdot \frac{|u|^{p+1}}{|x|} + \frac{1}{2} \cdot\frac{|\slashed{\nabla} u|^2}{|x|} \right) dxdt 
 =  \int_{\Rm^3} e_-(x,0) dx.
\] 
We then multiply both sides of this identity by $a_R(0)$, add it to \eqref{integral identity in weighted 1} and obtain
\begin{align*}
 & \pi \int_0^\infty a_R(t) d\mu(t) + \iint_{\Rm^3 \times \Rm^+} a_R(|x|+t)\left(\frac{p-1}{2(p+1)}\cdot \frac{|u|^{p+1}}{|x|} + \frac{1}{2}\cdot \frac{|\slashed{\nabla} u|^2}{|x|} \right) dxdt\\
 &\qquad =  \int_{\Rm^3} a_R(|x|) e_-(x,0) dx + \iint_{\Omega(0,R)} a'(|x|+t)\left[\frac{1}{p+1}|u|^{p+1} +  \frac{1}{2}\left|\slashed{\nabla} u\right|^2 \right] dxdt.
\end{align*}
Now we have the key observation by the assumption on the function $a$
\begin{align*}
 & \iint_{\Omega(0,R)} a'(|x|+t)\left[\frac{1}{p+1}|u|^{p+1} +  \frac{1}{2}\left|\slashed{\nabla} u\right|^2 \right]  dxdt\\
  \leq &  \iint_{\Omega(0,R)} \frac{\gamma a(|x|+t)}{|x|+t}\left[\frac{1}{p+1}|u|^{p+1} +  \frac{1}{2}\left|\slashed{\nabla} u\right|^2 \right]  dxdt\\
 \leq &   \iint_{\Omega(0,R)} \gamma a_R (|x|+t)\left[\frac{1}{p+1}\cdot \frac{|u|^{p+1}}{|x|} +  \frac{1}{2}\cdot \frac{\left|\slashed{\nabla} u\right|^2}{|x|} \right]  dxdt.
 \end{align*}
This immediately gives us
\begin{align*}
 & \pi \int_0^\infty a_R(t) d\mu(t) + \iint_{\Rm^3 \times \Rm^+} a_R(|x|+t)\left(\frac{p-1-2\gamma}{2(p+1)}\cdot \frac{|u|^{p+1}}{|x|} + \frac{1-\gamma}{2} \cdot\frac{|\slashed{\nabla} u|^2}{|x|} \right) dxdt\\
 &\qquad \leq  \int_{\Rm^3} a_R(|x|) e_-(x,0) dx \leq K_1.
\end{align*}
Finally we can make $R\rightarrow +\infty$ and finish our proof.
\end{proof}

\begin{remark} \label{p3 case of weighted Morawetz}
Let $p=3$ and $\gamma =1$. If a weight function $a(r)$ and a solution $u$ satisfy conditions in Proposition \ref{Weighted Morawetz General},
then we can follow the same argument as above and obtain
 \begin{align*}
   \pi \int_0^\infty a(t)\, d\mu(t) + \iint_{\Rm^3 \times \Rm^+} \frac{a(|x|+t)t}{|x|(|x|+t)}\left(\frac{1}{4}|u|^{4}+ \frac{1}{2}|\slashed{\nabla} u|^2 \right) dxdt \leq K_1.
 \end{align*}
\end{remark}

\begin{corollary}[Weighted Morawetz]\label{weighted Morawetz}
Assume that $3\leq p<5$ and $0<\kappa<1$. Let $u$ be a solution to (CP1) with a finite energy so that
\[
 K = \int_{\Rm^3} |x|^\kappa \left(\frac{1}{4}|(\mathbf{L}_+(u_0,u_1))(x) |^2 + \frac{1}{4}|\slashed{\nabla} u_0(x)|^2 + \frac{1}{2(p+1)}|u_0(x)|^{p+1} \right) dx < \infty,
\]
then we have the following weighted Morawetz estimates
\begin{align*}
  &\int_0^\infty t^{\kappa} d\mu(t) \lesssim_p K;& 
  &\int_0^\infty \int_{\Rm^3} \frac{(|x|+t)^\kappa \left(|u(x,t)|^{p+1}+|\slashed{\nabla} u(x,t)|^2\right)}{|x|} dxdt \lesssim_{p,\kappa} K.&
 \end{align*}
Here $\mu$ is the measure in the energy flux formula. In addition we have the decay estimate 
\[
 E_-(t)\lesssim_{p,\kappa} K t^{-\kappa}, \quad t>0.
\]
\end{corollary}
\begin{proof}
 We apply Proposition \ref{Weighted Morawetz General} with $a(r) = r^\kappa$, $\gamma =\kappa$ and $K_1=K$. This immediately gives us the weighted Morawetz estimates. In order to obtain the decay estimate we use the expression of inward energy \eqref{expression of E neg temp}
 \begin{align*}
   E_-(t)  & = \pi \mu([t,\infty)) + \int_t^\infty \int_{\Rm^3} \left(\frac{p-1}{2(p+1)}\frac{|u|^{p+1}}{|x|} + \frac{1}{2} \frac{|\slashed{\nabla} u|^2}{|x|} \right) dx dt' \\
   & \lesssim_p   t^{-\kappa} \int_t^\infty (t')^{\kappa}  d\mu(t') + t^{-\kappa} \int_t^\infty \int_{\Rm^3} \frac{(|x|+t')^\kappa \left(|u(x,t')|^{p+1}+|\slashed{\nabla} u(x,t')|^2\right)}{|x|} dxdt' \\
   & \lesssim_{p,\kappa} Kt^{-\kappa}.
 \end{align*}
\end{proof}

\section{Application on Scattering Theory}

In this section we give a scattering theory as an application of our inward/outward energy theory and weighted Morawetz estimates. The idea is to combine the decay estimates obtained above with the local theory. We first give an abstract scattering theory.

\begin{lemma} \label{abstract inter}
Let $3\leq p \leq 5$. Assume that the constants $1<q_1,r_1,q_2,r_2<\infty$ and $k_1,k_2>0$ satisfy
\begin{align*}
 &1/q_2 + 3/r_2 = 1/2;& &k_1 +k_2 = p-1;&\\
 &k_1/q_1 +k_2/q_2 = 1/2;& &k_1/r_1+k_2/r_2 = 1/2.&  
\end{align*}
 If $u$ is a finite-energy solution to (CP1) with $\|u\|_{L^{q_1} L^{r_1}(\Rm^+ \times \Rm^3)} < + \infty$, then we also have
 \[
    \|u\|_{L^{2(p-1)} L^{2(p-1)} (\Rm^+ \times \Rm^3)}< +\infty.
 \]
In addition, the solution scatters in the space $\dot{H}^1 \times L^2(\Rm^3)$ in the positive time direction. More precisely, there exists $(v_0^+,v_1^+) \in \dot{H}\times L^2(\Rm^3)$, so that 
\[
 \lim_{t \rightarrow +\infty} \left\|\begin{pmatrix}u(\cdot,t)\\ \partial_t u(\cdot,t)\end{pmatrix} - \mathbf{S}_L (t) \begin{pmatrix} v_0^+\\ v_1^+\end{pmatrix}\right\|_{\dot{H}^{1} \times L^2} = 0.
\]
\end{lemma}  
\begin{proof}
In this proof we will use the notation $F(u) = -|u|^{p-1} u$. First of all, we can apply Strichartz estimates and conclude that for any $T>0$
 \begin{align*}
 \|u\|_{L^{q_2} L^{r_2}([0,T]\times \Rm^3)} + & \|D_x^{1/2} u\|_{L^4 L^4 ([0,T]\times \Rm^3)}\\
   & \lesssim \|(u_0,u_1)\|_{\dot{H}^1 \times L^2(\Rm^3)} + \|F(u)\|_{L^1 L^2 ([0,T]\times \Rm^3)}<+\infty.
 \end{align*}
As a result, if we fix a time $t_0>0$, then the function defined by
\[
 g_{t_0}(T) = \|u\|_{L^{q_2} L^{r_2}([t_0,T]\times \Rm^3)} +  \|D_x^{1/2} u\|_{L^4 L^4 ([t_0,T]\times \Rm^3)}
\]
 is always a continuous real-valued function of $T\in [t_0,\infty)$. Furthermore, we can apply Strichartz estimates and obtain 
 \begin{align*}
  g_{t_0}(T) \lesssim \|(u(\cdot,t_0),\partial_t u(\cdot,t_0))\|_{\dot{H}^1 \times L^2} + \|D_x^{1/2} F(u)\|_{L^{4/3} L^{4/3}([t_0,T]\times \Rm^3)}.
 \end{align*}
 We may apply the energy conservation law and the ``chain rule'' Lemma \ref{chain rule} here:
 \[
  g_{t_0}(T) \lesssim E^{1/2} + \|D_x^{1/2} u\|_{L^4 L^4([t_0,T]\times \Rm^3)} \|F'(u)\|_{L^2 L^2([t_0,T]\times \Rm^3)}.
 \] 
 The assumptions on the constants $p_1,r_1,p_2,r_2,k_1,k_2$ guarantee that the following H\"{o}lder inequality holds
 \begin{equation} \label{Holder Inequality}
 \|F'(u)\|_{L^2 L^2} \lesssim_1 \||u|^{p-1}\|_{L^2 L^2} \leq \|u\|_{L^{q_1}L^{r_1}}^{k_1} \|u\|_{L^{q_2} L^{r_2}}^{k_2}.
\end{equation}
 Thus we obtain an inequality about $g_{t_0}$
 \begin{align*}
  g_{t_0}(T) & \lesssim E^{1/2} + \|D_x^{1/2} u\|_{L^4 L^4([t_0,T]\times \Rm^3)} \|u\|_{L^{p_1}L^{r_1}([t_0,T]\times \Rm^3)}^{k_1} \|u\|_{L^{q_2} L^{r_2}([t_0,T]\times \Rm^3)}^{k_2} \\
  & \lesssim E^{1/2} + \|u\|_{L^{q_1}L^{r_1}([t_0,T]\times \Rm^3)}^{k_1} \left(g_{t_0}(T)\right)^{k_2+1}.
 \end{align*}
 In other words, there exists a constant $C$ independent of $t_0, T$, so that
 \begin{equation} \label{self inequality}
  g_{t_0}(T) \leq CE^{1/2} + C\|u\|_{L^{p_1}L^{r_1}([t_0,T]\times \Rm^3)}^{k_1} \left(g_{t_0}(T)\right)^{k_2+1}.
 \end{equation}
 We may choose a small constant $\eps>0$, so that 
 \begin{equation} \label{assumption of eps}
  2CE^{1/2} > CE^{1/2} + C\eps^{k_1} \left(2CE^{1/2}\right)^{k_2+1}.
 \end{equation}
 Since $\|u\|_{L^{q_1}L^{r_1}(\Rm^+ \times \Rm^3)} < +\infty$ and $q_1<+\infty$, we can always find a time $t_0$ so that 
 \[
  \|u\|_{L^{q_1}L^{r_1}([t_0,\infty) \times \Rm^3)} < \eps.
 \] 
Now we can simply compare the inequalities \eqref{self inequality} and \eqref{assumption of eps} to verify that $g_{t_0}(T) \neq 2CE^{1/2}$ for all $T>t_0$. Next we observe $g_{t_0}(t_0) = 0$, apply an argument of continuity and conclude $g_{t_0}(T) < 2CE^{1/2}$ for all $T\geq t_0$. This uniform upper bound implies 
\begin{align*}
 \|u\|_{L^{q_2} L^{r_2}([t_0,\infty)\times \Rm^3)} & + \|D_x^{1/2} u\|_{L^4 L^4 ([t_0,\infty)\times \Rm^3)}\leq 2CE^{1/2}\\
 &  \Rightarrow  \|u\|_{L^{q_2} L^{r_2}(\Rm^+ \times \Rm^3)}  + \|D_x^{1/2} u\|_{L^4 L^4 (\Rm^+ \times \Rm^3)} < +\infty.
\end{align*}
 Next we recall the inequality \eqref{Holder Inequality} to conclude
\begin{align*}
 \|u\|_{L^{2(p-1)}L^{2(p-1)}(\Rm^+\times \Rm^3)}^{p-1} &  = \||u|^{p-1}\|_{L^2 L^2 (\Rm^+\times \Rm^3)} \\
 & \leq \|u\|_{L^{q_1}L^{r_1}(\Rm^+\times \Rm^3)}^{k_1} \|u\|_{L^{q_2} L^{r_2}(\Rm^+\times \Rm^3)}^{k_2} < +\infty.
\end{align*}
Thus $\|u\|_{L^{2(p-1)} L^{2(p-1)}(\Rm^+ \times \Rm^3)} < +\infty$. The inequalities above also imply
\[
 \|D_x^{1/2} F(u)\|_{L^{4/3} L^{4/3}(\Rm^+ \times \Rm^3)} \lesssim \|D_x^{1/2}\|_{L^4 L^4(\Rm^+ \times \Rm^3)} \|F'(u)\|_{L^2 L^2(\Rm^+ \times \Rm^3)} < +\infty.
\]
Finally we recall $\mathbf{S}_L(t)$ preserve the $\dot{H}^1 \times L^2$ norm, apply the Strichartz estimate to obtain
\begin{align*}
 & \limsup_{t_1,t_2\rightarrow +\infty} \left\|\mathbf{S}_L (-t_1)\begin{pmatrix}u(\cdot,t_1)\\ \partial_t u(\cdot,t_1)\end{pmatrix} - \mathbf{S}_L (-t_2) \begin{pmatrix}u(\cdot,t_2)\\ \partial_t u(\cdot,t_2)\end{pmatrix}\right\|_{\dot{H}^{1} \times L^2 (\Rm^3)}\\
 = & \limsup_{t_1,t_2\rightarrow +\infty} \left\|\mathbf{S}_L (t_2-t_1)\begin{pmatrix}u(\cdot,t_1)\\ \partial_t u(\cdot,t_1)\end{pmatrix} - \begin{pmatrix}u(\cdot,t_2)\\ \partial_t u(\cdot,t_2)\end{pmatrix}\right\|_{\dot{H}^{1} \times L^2 (\Rm^3)}\\
 \lesssim & \limsup_{t_1,t_2\rightarrow +\infty} \|D_x^{1/2} F(u)\|_{L^{4/3} L^{4/3}([t_1,t_2] \times \Rm^3)}\\
  = & 0.
\end{align*}
Since $\dot{H}^1 \times L^2$ is a complete space, there exists $(v_0^+, v_1^+)\in \dot{H}^1 \times L^2$ so that
\[
 \left\|\mathbf{S}_L (-t)\begin{pmatrix}u(\cdot,t)\\ \partial_t u(\cdot,t)\end{pmatrix} - \begin{pmatrix} v_0^+\\ v_1^+\end{pmatrix}\right\|_{\dot{H}^1 \times L^2} \rightarrow 0\; \Rightarrow\; \left\|\begin{pmatrix}u(\cdot,t)\\ \partial_t u(\cdot,t)\end{pmatrix} - \mathbf{S}_L (t) \begin{pmatrix} v_0^+\\ v_1^+\end{pmatrix}\right\|_{\dot{H}^1 \times L^2} \rightarrow 0.
\]
 \end{proof}
\begin{remark}
 Lemma \ref{abstract inter} also holds for $q_1 = +\infty$ as long as we substitute the assumption $\|u\|_{L^{q_1} L^{r_1}(\Rm^+ \times \Rm^3)} < +\infty$ by $\displaystyle \lim_{t \rightarrow +\infty} \|u\|_{L^\infty L^{r_1}([t,\infty)\times \Rm^3)} = 0$. We may combine the inward/outward energy method and this abstract scattering theory to give a new proof of the following already known result: Any global-in-time solution $u$ to defocusing, energy critical 3D wave equation with a finite energy must scatter in both two time directions. The proof consists of two major steps
 \begin{itemize}
  \item We apply inward/outward energy method to show $\displaystyle \lim_{t \rightarrow \pm \infty} E_{\mp} (t) = 0$. This implies that $\displaystyle \lim_{t \rightarrow \pm \infty} \|u(\cdot,t)\|_{L^6(\Rm^3)}=0$.
  \item We apply Lemma \ref{abstract inter} with $p=5$, $(q_1,r_1) = (\infty, 6)$, $(q_2,r_2)=(4,12)$ and $(k_1,k_2) = (2,2)$.
 \end{itemize}
\end{remark}

\subsection{Proof of Theorem \ref{main 2}} \label{sec: proof main 2}
Since the wave equation is time-reversible, we only need to prove the scattering in the positive time direction. Let $u$ and its initial data $(u_0,u_1)$ be as in Theorem \ref{main 2}. According to Remark \ref{relationship of u w energy}, we also have
\begin{equation}
  K = \int_{\Rm^3} |x|^\kappa \left(\frac{1}{4}|\mathbf{L}_+ (u_0,u_1)|^2 + \frac{1}{4}|\slashed{\nabla} u_0|^2 + \frac{1}{2(p+1)}|u_0|^{p+1}\right) dx < +\infty. \label{condition for weighted morawetz}
\end{equation}
Now we may apply Corollary \ref{weighted Morawetz} and obtain
\[
 E_-(t) \lesssim_{p,\kappa} K t^{-\kappa},\; t>0; \quad \Rightarrow \quad \|u(\cdot,t)\|_{L^{p+1}(\Rm^3)}^{p+1} \lesssim_{p,\kappa} K t^{-\kappa}, \; t>0.
\]
This means $u \in L^{q} L^{p+1}([1,\infty)\times \Rm^3)$ for any $q > (p+1)/\kappa$. We also have $u \in L^\infty L^{p+1} (\Rm \times \Rm^3)$ by the energy conservation law. As a result we have 
\begin{equation}
 u \in L^{q} L^{p+1}(\Rm^+ \times \Rm^3), \quad \forall q > (p+1)/\kappa.  \label{finite norm}
\end{equation}
Let us choose a constant $\beta$ so that
\[
  \max\left\{\frac{5-p}{2\kappa},4-p\right\} < \beta < 1,
\]
and then choose $q_1, r_1, q_2, r_2, k_1, k_2$ accordingly 
\begin{align*}
 &\frac{1}{q_1} = \frac{5-p}{2(p+1)\beta};& &\frac{1}{r_1} = \frac{1}{p+1};&\\
 &\frac{1}{q_2} = \frac{\beta + (p-4)}{2(p-1)-4\beta}>0;& &\frac{1}{r_2} = \frac{1-\beta}{2(p-1)-4\beta}>0;&\\
  &k_1=\frac{(p+1)\beta}{1+\beta}>0;& &k_2 = \frac{(p-1)-2\beta}{1 + \beta}>0.&
\end{align*}
One can check that the following identities hold
\begin{align*}
 &\frac{1}{q_2} + \frac{3}{r_2} = \frac{1}{2};& &k_1+k_2 = p-1;&  &\frac{k_1}{q_1} + \frac{k_2}{q_2} = \frac{1}{2};& &\frac{k_1}{r_1} + \frac{k_2}{r_2} = \frac{1}{2}.&
\end{align*}
Our assumption on $\beta$ guarantees $q_1 = \frac{2(p+1)\beta}{5-p} > (p+1)/\kappa$. Thus we have $u \in L^{q_1} L^{r_1}(\Rm^+ \times \Rm^3)$ by \eqref{finite norm}. Now we can apply Lemma \ref{abstract inter} with these constants $q_1, r_1, q_2, r_2, k_1, k_2$ and the solution $u$ to conclude 
\begin{itemize}
 \item[(a)] $u \in L^{2(p-1)} L^{2(p-1)} (\Rm^+ \times \Rm^3)$. 
 \item[(b)] The solution $u$ scatters in the space $\dot{H}^1 \times L^2$ as $t \rightarrow +\infty$.
\end{itemize}
According to Proposition \ref{local existence}, conclusion (a) implies that the scattering also happens in the space $\dot{H}^{s_p} \times \dot{H}^{s_p-1}$. 
This then gives the scattering of solution $u$ in the spaces $\dot{H}^s \times \dot{H}^{s-1}(\Rm^3)$ for all $s \in [s_p,1]$ by Lamma \ref{scattering in different spaces}.

\subsection{Proof of Theorem \ref{main 3}}
There is a technical difficulty in the proof. We do not assume that the initial data come with a finite energy. This can be solved by approximation techniques. Let us assume $u$ has a maximal lifespan $(-T_-,T_+)$. We recall the following smooth cut-off operator introduced in Lemma \ref{cutoff lemma}
\begin{align*}
 &\mathbf{P}_{\eps} f= \phi_\eps\cdot f,& &\phi_\eps(x) = \phi(x/\eps),&
\end{align*}
where $\phi: \Rm^3 \rightarrow [0,1]$ is a radial, smooth cut-off function satisfying 
\[
  \phi(x) = \left\{\begin{array}{ll}1, & \hbox{if}\; |x|\geq 1;\\ 0, & \hbox{if}\; |x|<1/2; \end{array}\right.
 \] 
and define $(u_0^\eps, u_1^\eps) = (\mathbf{P}_\eps u_0,\mathbf{P}_\eps u_1)$. A straightforward calculation shows
\begin{align*}
 \int_{\Rm^3} |x|\cdot|\nabla u_0^\eps|^2 dx & = \int |x|\cdot |\phi_\eps(x) \nabla u_0 + u_0 \nabla \phi_\eps|^2 dx \\
 & \leq (1+\delta)\int_{\Rm^3} |x|\cdot |\phi_\eps \nabla u_0|^2 dx + \left(1+\frac{1}{\delta}\right) \int_{\Rm^3}|x|\cdot |u_0 \nabla \phi_\eps|^2 dx \\
 & \leq (1+\delta)\int_{\Rm^3} |x|\cdot |\nabla u_0|^2 dx + C\left(1+\frac{1}{\delta}\right) \int_{\frac{\eps}{2}<|x|<\eps} \eps^{-1}\cdot |u_0(x)|^2 dx.
\end{align*}
Here $\delta >0$ is an arbitrary constant. By Cauchy-Schwartz inequality we have
\begin{align*}
 \int_{\frac{\eps}{2}<|x|<\eps} \eps^{-1} \cdot |u_0(x)|^2 dx & \leq 
 \left( \int_{\frac{\eps}{2}<|x|<\eps} |x|\cdot |u_0(x)|^4 dx\right)^{1/2}\left( \int_{\frac{\eps}{2}<|x|<\eps} \frac{\eps^{-2}}{|x|} dx\right)^{1/2}\\
 & \lesssim_1 \left( \int_{\frac{\eps}{2}<|x|<\eps} |x|\cdot |u_0(x)|^4 dx\right)^{1/2} \rightarrow 0, \quad \hbox{as}\; \eps \rightarrow 0^+.
\end{align*}
Thus we have
\[
 \limsup_{\eps\rightarrow 0^+} \int_{\Rm^3} |x|\cdot|\nabla u_0^\eps|^2 dx \leq \int_{\Rm^3} |x|\cdot |\nabla u_0|^2 dx.
\]
This implies 
\begin{align*}
 \limsup_{\eps\rightarrow 0^+} \int_{\Rm^3} |x|\left(\frac{1}{2}|\nabla u_0^\eps|^2 + \frac{1}{2}|u_1^\eps|^2 + \frac{1}{4}|u_0^\eps|^4 \right) dx 
 & \leq \int_{\Rm^3} |x|\left(\frac{1}{2}|u_0|^2 + \frac{1}{2}|u_1|^2 + \frac{1}{4}|u_0|^4\right) dx\\
 & = E_{1,0}(u_0,u_1) < +\infty.
\end{align*}
Since the support of $(u_0^\eps,u_1^\eps)$ is contained in $\{x: |x|\geq \eps/2\}$, the inequality above means $E(u_0^\eps,u_1^\eps)< +\infty$. By Remark 
\ref{relationship of u w energy}, we also have 
\[
  \limsup_{\eps\rightarrow 0^+} \int_{\Rm^3} |x|\left(\frac{1}{4}|\mathbf{L}_+ (u_0^\eps,u_1^\eps)|^2 + \frac{1}{4}|\slashed{\nabla} u_0^\eps|^2 + \frac{1}{8}|u_0^\eps|^4 \right) dx \leq E_{1,0}(u_0,u_1).
\]
Now we are able to apply Remark \ref{p3 case of weighted Morawetz} with $a(r)=r$ and conclude that the corresponding solution $u^\eps$ to (CP1) with initial data $(u_0^\eps,u_1^\eps)$ satisfies the weighted Morawetz estimate
\[
 \limsup_{\eps\rightarrow 0^+} \iint_{\Rm^3 \times \Rm^+} \frac{t}{|x|}\left(\frac{1}{4}|u^\eps(x,t)|^4 + \frac{1}{2}|\slashed{\nabla}u^\eps (x,t)|^2 \right) dx dt \leq E_{0,1}(u_0,u_1).
\]
Now let us consider the limit process $\eps \rightarrow 0^+$. According to Lemma \ref{cutoff lemma}, we have $\|(u_0^\eps,u_1^\eps)-(u_0,u_1)\|_{\dot{H}^{1/2} \times \dot{H}^{-1/2}} \rightarrow 0$. This implies that $u^\eps(x,t)$ converges to $u(x,t)$ almost everywhere in $\Rm^3 \times [0,T_+)$, possibly up to a subsequence, by the continuous dependence of solutions on initial data, as given in Proposition \ref{continuous dependence}. By Fatou's lemma we have
\[
 \int_0^{T_+} \int_{\Rm^3} \frac{t}{|x|} |u(x,t)|^4  dx dt \leq 4E_{1,0}(u_0,u_1).
\]
In addition, Corollary \ref{tail estimate} guarantees that for sufficiently large radius $R$, we always have 
\[
 \int_0^{T_+} \int_{|x|>R+t} |u(x,t)|^4  dx dt < +\infty.
\]
We may fix such a radius $R$ and an arbitrary time $T \in (0,T_+)$, use both two estimates above, then obtain 
\begin{align*}
 \int_{T}^{T_+} \!\!\int_{\Rm^3} |u(x,t)|^4  dx dt & = \int_T^{T_+} \!\!\int_{|x|>R+t} |x(x,t)|^4 dx dt + \int_T^{T_+} \!\!\int_{|x|<R+t} |x(x,t)|^4 dx dt \\
 & \leq \int_T^{T_+} \!\!\int_{|x|>R+t} |x(x,t)|^4 dx dt + \frac{R+T}{T}\!\!\int_T^{T_+} \int_{|x|<R+t} \frac{t}{|x|} |x(x,t)|^4 dx dt\\
 & < +\infty.
\end{align*}
This immediately gives us the scattering in the positive time direction by the scattering criterion given in Proposition \ref{local existence}. The negative time direction can be dealt with in the same manner since the wave equation is time-reversible. 

\end{document}